\documentclass[10pt, a4paper]{amsart}
\usepackage[utf8]{inputenc}
\usepackage[hidelinks, pdfencoding=auto]{hyperref}
\usepackage{optidef}    % typesetting of optimization problems
\usepackage{array}      % spacing in arrays
\usepackage{tikz}       % pretty pictures
\usetikzlibrary{
    positioning,        % for placing nodes relative to each other
    arrows,             % stealth arrow tip
    calc,               % for perspective calculation for 3d cube
	shapes,
    decorations.markings,
    trees,
    graphs,
}
\usepackage{tikzscale}
\usepackage{bm}         % for bold Greek letters
\usepackage{framed}     % for vertical line on the left
\usepackage{adjustbox}
\usepackage{xcolor}     % for shaded headers
\usepackage{booktabs}   % for solution tables
\usepackage{longtable}  % for solution tables
\usepackage{float}      % for H option
\usepackage{multirow}   % for square example 
\usepackage{siunitx}    % to make long numbers easier to read
\usepackage{caption}    % for side by side figures
\usepackage{subcaption} % for side by side figures
\usepackage{mathtools}  % for \underbrace

\usepackage{todonotes}  % temporary
\usepackage{svg}        % temporary
\renewenvironment{leftbar}{\MakeFramed {\advance\hsize-\width \FrameRestore}}{\endMakeFramed}

\newcounter{lpcounter}
\renewcommand\thelpcounter{\Roman{lpcounter}}

\newenvironment{lp}[1]{%
\vspace{1mm}\pagebreak[1]\refstepcounter{lpcounter}\noindent\adjustbox{bgcolor=black!5,minipage=[t]{\linewidth}}{\hrule\vspace{1mm}
\noindent{\bf\footnotesize Linear Problem \thelpcounter.} {\footnotesize #1.}
\vspace{1mm}\hrule}\vspace{1mm}

\begin{leftbar}\noindent\ignorespaces}
{\nopagebreak\end{leftbar}\nopagebreak\vspace{1mm}\hrule\pagebreak[1]\vspace{3mm}}

\newcounter{opcounter}
\renewcommand\theopcounter{\Roman{opcounter}}

\newenvironment{op}[1]{%
\vspace{1mm}\pagebreak[1]\refstepcounter{opcounter}\noindent\adjustbox{bgcolor=black!5,minipage=[t]{\linewidth}}{\hrule\vspace{1mm}
\noindent{\bf\footnotesize Optimization Problem \theopcounter.} {\footnotesize #1.}
\vspace{1mm}\hrule}\vspace{1mm}

\begin{leftbar}\noindent\ignorespaces}
{\nopagebreak\end{leftbar}\nopagebreak\vspace{1mm}\hrule\pagebreak[1]\vspace{3mm}}

\newcommand{\overbar}[1]{\mkern 1.5mu\overline{\mkern-1.5mu#1\mkern-1.5mu}\mkern 1.5mu}

\newtheorem{theorem}{Theorem}[section]
\newtheorem{lemma}[theorem]{Lemma}

\newtheorem{corollary}[theorem]{Corollary}
\newtheorem{proposition}[theorem]{Proposition}

\newtheorem{definition}[theorem]{Definition}

\theoremstyle{definition}

\DeclareMathOperator{\diam}{diam}
\DeclareMathOperator{\supp}{supp}

\DeclareMathOperator{\minp}{min_{>0}}

\newcommand{\df}[1]{{\bf #1}}
\newcommand{\fs}{\scriptscriptstyle}

\definecolor{bbord}{RGB}{70,78,82}     % node border
\definecolor{ffill}{RGB}{171,213,238}  % node fill
\definecolor{eedge}{RGB}{27,128,196}   % major edge
\tikzset{ 
    edge/.style={
        draw=black,
        line width=0.5pt,
        ->,             % alwways draw arrow tip
        >=stealth,      % style of arrow tip
        shorten >=2pt,  % shorten a bit, so that it doesn't quite
        shorten <=2pt,  % touch the nodes
        preaction={     % provide a wider white background for
          draw=white,   % each arrow for the intersection effect
          line width=3pt,
          -,            % no arrow tip for background
          }
    },
    vertex/.style={
        anchor=center,
        text=black,
        inner sep=5pt,
        shape=circle,
        draw=black,          % border
        fill=white,          % background
        font=\footnotesize, 
        minimum height=12pt  % assign minimum height to make nodes equally
    },
    hidden/.style={
        draw=none,
        fill=none
    },
    weight/.style={
      font=\footnotesize
    }
}

\title{Property A and duality in linear programming}
\date{\today}
\author{G.~C.~Bell}
\address{Mathematics and Statistics\\ UNC Greensboro\\ Greensboro, NC 27402 USA}
\email{gcbell@uncg.edu}

\author{A.~Nagórko}
\address{Department of Mathematics\\ University of Warsaw\\ ul. Stefana Banacha 2\\ 02-097 Warszawa, Poland}
\email{amn@mimuw.edu.pl}

\subjclass[2020]{Primary: 90C05; Secondary: 90C46; 20F69; 05C48}
\keywords{Property A; linear programming; strong duality; Cheeger constant; expanders}

\begin{document}

\begin{abstract}
Property A is a form of weak amenability for groups and metric spaces introduced as an approach
  to the famous Novikov higher signature conjecture, one of the most important unsolved problems in topology.

We show that property A can be reduced to a sequence of linear programming optimization problems
  on finite graphs.
We explore the dual problems, which turn out to have interesting interpretations as combinatorial problems concerning the maximum total supply of flows on a network.

Using isoperimetric inequalities, we relate the dual problems to the Cheeger constant of the graph and
explore the role played by symmetry of a graph to obtain a striking characterization of the difference between an expander and a graph without property A.
Property A turns out to be a new measure of connectivity of a graph that is relevant to graph theory.

The dual linear problems can be solved using a variety of methods, which we demonstrate
  on several enlightening examples.
As a demonstration of the power of this
linear programming approach we give elegant proofs of theorems of Nowak and Willett about
  graphs without property A.

\end{abstract}
\maketitle

\section{Introduction}

% Why do we care (about Property A)?

Property~A is a form of weak amenability for groups and metric spaces introduced by G.~Yu~\cite{yu2000} at the turn of the century as an approach to the famous Novikov higher signature conjecture. First stated over fifty years ago, the Novikov conjecture asserts that the so-called higher signatures determined by a discrete group $\Gamma$ are homotopy invariant and remains one of the most important unsolved problems in topology~\cite{frr1995}.

Yu proved that a discrete metric space $X$ with property A admits a coarsely uniform embedding into Hilbert space and that this implies the coarse Baum-Connes conjecture. In the case that $X$ is the Cayley graph of a finitely generated group~$\Gamma$, this implies the Novikov conjecture for $\Gamma$. 

A great deal of attention has been devoted to the study of property~A because of the number of large classes of spaces known to have property~A and the robust collection of operations on groups and spaces under which property~A is closed. 
Among the properties of metric space that imply property~A are
%Property A is the focal point in the above chain of implications.
%Many invariants were developed that imply Property A, e.g.
  finite asymptotic dimension~\cite{higson-roe2000},
  finite decomposition complexity~\cite{guentner-tessera-yu2012},
  asymptotic property C~\cite{dranish2000}, and
  polynomial dimension growth~\cite{dranish2006}.
Moreover property~A is known to satisfy more so-called permanence properties than, for example, finite asymptotic dimension, the rate of dimension growth, or uniform embeddability in Hilbert space~\cite{guentner2014}.

While property~A is defined for metric spaces, the most important use case
  is that of Cayley graphs of countable groups.
For a graph $G$ and a nonnegative number $S\ge 0$ called the scale, 
we define $\varepsilon_{G, S}$
  to be the minimal variation of probability measures on $G$ that are supported at scale $S$ (see Optimization Problem~\ref{op:measures}).
The graph $G$ has property~A if and only if $\lim_{S \to \infty} \varepsilon_{G, S} = 0$, i.e. if the minimal variation of probability measures tends to zero when the scale goes to infinity.

This paper stems from the observation that
  $\varepsilon_{G, S}$ can be computed with linear programming.
We use a standard trick (Theorem~\ref{thm:property A limit of finite graphs})
  to reduce our considerations to finite graphs so we can
  apply linear optimization to the study of property~A.

Beginning with the linear problem corresponding to the minimal variation of 
probability measures $\varepsilon$, we construct its dual problem. % to the minimal variance $\varepsilon$ of probability measures.
This dual problem has a interesting interpretation as a combinatorial problem concerning the maximum supply $\sum_i \eta_i$ of flows on a network.
By strong duality, we have $\varepsilon = \sum_i \eta_i$.

There are many positive results about property~A such as the permanence properties alluded to above.
This has to do with the fact that to prove that $\varepsilon_{G, S} \to 0$
  it is enough to find sufficiently nice upper bounds on $\varepsilon_{G, S}$.
Computing $\varepsilon_{G, S}$ is a minimization problem so upper bounds are easy to get (e.g. any admissible solution of the problem provides an upper bound). 
On the other hand, constructions of graphs that fail to have property~A % (e.g. constructions of graphs without Property A) 
are scarce and either rely on something stronger (such as failure of the space to embed in Hilbert space) or rely on ad hoc methods. Indeed, to show that $\lim \varepsilon_{G, S} > 0$ we need information about the
  optimal value of $\varepsilon_{G, S}$, which is much harder to determine.
But, when we pass to the dual problem, which concerns maximizing $\sum_i \eta_i$, lower bounds
  are easy to get.
The open status of the Novikov conjecture means that negative results concerning property~A are particularly enticing.
Using the dual problem, we characterize the gap between the Cheeger constant and variation of probability measures to show the exact place where such negative examples may be found.
  
As an application of the dual problem, we compute the values of $\varepsilon$ (minimal variation of probability measures) for $n$-dimensional 
cube graphs and for graphs with large girth.
These computations are effortless and as corollaries we obtain a theorem of Nowak~\cite{nowak2007} that a disjoint union of $n$-cubes does not have property~A and a theorem of Willett~\cite{willett2011} that a disjoint union of regular graphs with increasing girth does not have property~A.

We refer to Figure~\ref{fig: preliminary primal and dual problems} as we describe the flow and contents of the present paper. 
Examining the explicit computations in the dual problem (Pseudo Flows (II)) leads us to a natural relaxation in terms of Isoperimetric Inequalities (III). 
We pass to the dual problem to Isoperimetric Inequalities (III), which is again on the side of primal/minimization problems, and show that it is equivalent to a well-known reformulation of property~A in the language of partitions of unity (see Partitions of Unity (IV)).
We show more: the relaxation (III) is in fact the projection of the dual problem (II)
  into a lower dimensional space.
In other words, on the dual side, the ``probability measures'' formulation is a lift of the ``partitions of unity'' formulation of property~A.
Both dual problems are interesting on their own: the dual to ``probability measures'' has more variables (but is polynomial in size with respect to the size of the graph),
while the projection---the dual to ``partitions of unity''---has fewer variables, but an exponential number of constraints.

{
	% https://tex.stackexchange.com/questions/346106/creating-cube-with-tikz
	% https://tex.stackexchange.com/questions/8890/tikz-how-to-draw-boxes-around-set-of-nodes
	\tikzset{%
		node distance=15mm,
		lp/.style={
			rectangle,
			draw=gray,
			thick,
			fill=gray!20,
			text width=25mm,
			align=center,
			% rounded corners,
			minimum height=1cm
		},
		every node/.style={%
			align=center,
			fill=white,
			font=\small
		},
		every edge/.append style={every node/.style={font=\footnotesize, fill=white}}
	}
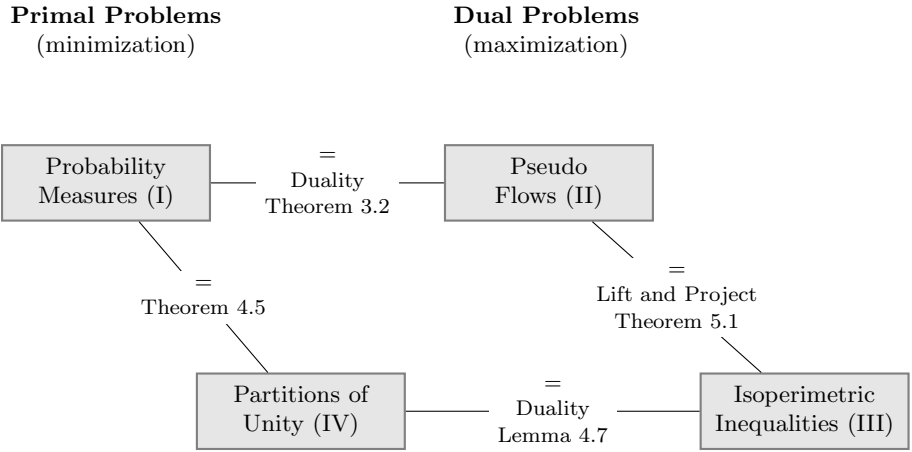
\begin{figure}[H]
		\begin{center}
			\begin{tikzpicture}
				\node (p) {\bf Primal Problems\\ (minimization)};
				\node[right=28mm of p] (d) {\bf Dual Problems\\ (maximization)};
				
				\node[lp, below=2cm of p.west, anchor=west] (p1) {Probability Measures (\ref{lp:measures})};
				\node[lp, below right=2cm and -2mm of p1] (p2) {Partitions of Unity (\ref{lp:partition})};
				
				\node[lp, below=2cm of d.west, anchor=west] (d1) {Pseudo Flows (\ref{lp:pseudo-flows})};
				\node[lp, below right=2cm and 6mm of d1] (d2) {Isoperimetric Inequalities (\ref{lp:isoperimetric})};
				\path (p1) edge node[align=center] {$=$\\ Theorem~\ref{thm:measures equals partition}} (p2);
				\path (p1) edge node[align=center] {$=$\\ Duality\\ Theorem~\ref{thm:duality}} (d1);
				\path (d1) edge node[align=center] {$=$\\ Lift and Project\\ Theorem~\ref{thm:lift and project}} (d2);
				
				\path (p2) edge node[align=center] {$=$\\ Duality\\ Lemma~\ref{lem:isoperimetric dual to partition}} (d2);
\end{tikzpicture}
		\end{center}
		\caption{Exploring four problems in the primal and dual realms.}
		\label{fig: preliminary primal and dual problems}
	\end{figure}
}

The set of optimal solutions to a linear problem is a convex polyhedron.
An action on the graph $G$ by isometries gives rise to an action on the vertices of this polyhedron
 and to additional symmetry constraints that are satisfied by the solution that is in its center of mass. 
If the action is transitive on both edges and vertices, 
  then for the Isoperimetric Inequalities problem
  these symmetry constraints reduce the problem to a single variable $\eta$
  and it follows directly from the problem definition that its optimal 
  solution is equal to $\frac{|V|}{|E|} \gamma(G, S)$, where $\gamma(G, S)$ is the Cheeger constant of $G$ at scale $S$ (cf. Definition~\ref{def:cheeger at scale S}).

This shows two things.
If the graph is symmetric enough, then 
  there is no difference between the expanding property (Cheeger constant $\geq \delta$ at all scales, for some $\delta > 0$) and lack of property~A.
Note that the symmetry must be present in finite convex subgraphs of the graph:
  infinite trees have large Cheeger constants at a fixed scale, yet they satisfy property~A.
This agrees with approaches currently known in the literature, where
  groups without property~A (Gromov monster groups~\cite{gromov2003})
  are constructed by showing that they contain expanders in their Cayley graphs.

If the graph is not symmetric, then we can still put symmetry constraints
  into the linear programming formulation to compute the Cheeger constant at scale $S$ for the graph, see Figure~\ref{fig: primal and dual problems}.
We study this problem, which we call the Uniform Flows problem (\ref{lp:uniform}).
Since it is a constrained dual (maximization) problem, the Cheeger constant will be smaller than the variation of probability measures $\varepsilon$.
To understand the difference between these two, we constructed the dual problem to the Cheeger constant problem to express it in the language of probability measures.
It turns out that it is equivalent to a problem, which we call Mean Property~A (\ref{lp:mean}), in which the functionals are probability measures only on average and their variation is bounded by $\varepsilon$ only on average.
Hence the expanding property of a graph contradicts not only property~A, but also this much weaker mean property~A (see Theorem~\ref{thm:mean A}), leaving plenty of room for negative examples that are not expanders.

In the paper we provide the details of many examples, computed with different techniques.
With the exception of the experimental results (which are not the topic of this paper),
  none of these are computed with the simplex method.

There is a notable difference between the Cheeger constant at scale $S$ and the regular Cheeger constant of the graph.
Using a lift of the linear formulation of the Cheeger constant at scale $S$ problem to
  the problem about pseudo flows we get a polynomial time algorithm for its computation, while computation of the Cheeger constant is NP-hard~\cite{leightonrao1999}.

This paper opens new areas of research.
The first area is in geometric group theory, where some of the main interest in property~A lies. One could ask, for example: 
How big is the gap between the expanding property of a graph (i.e. the Cheeger constant)
and property~A? 
How can this gap be exploited to construct new examples of graphs without property~A?
This question can be approached from both the theoretical and the experimental angle.
For example, for small values of $n$ and $S$, computations indicate that there is a large discrepancy between the expected value of the Cheeger constant and of the expected value of the variation of probability measures for random graphs on $n$ vertices with edge density $d$ at scale~$S$.

The second area is in graph theory.
Variation of probability measures estimates graph connectivity in a manner similar to the Cheeger constant, but is more resilient to local defects of the graph, e.g. it is not ruined by a single isolated vertex (which immediately makes the Cheeger constant zero).
How is this new connectivity measure relevant to graph theory?
Also in the opposite direction, in what ways can the wealth
 of knowledge about the Cheeger constant be transferred into geometric group theory?

{
	% https://tex.stackexchange.com/questions/346106/creating-cube-with-tikz
	% https://tex.stackexchange.com/questions/8890/tikz-how-to-draw-boxes-around-set-of-nodes
	\tikzset{%
		node distance=15mm,
		lp/.style={
			rectangle,
			draw=gray,
			thick,
			fill=gray!20,
			text width=25mm,
			align=center,
			% rounded corners,
			minimum height=1cm
		},
		every node/.style={%
			align=center,
			fill=white,
			font=\small
		},
		every edge/.append style={every node/.style={font=\footnotesize, fill=white}}
	}
	
	\begin{figure}[t]
		\begin{center}
			\begin{tikzpicture}
				\node (p) {\bf Primal Problems\\ (minimization)};
				\node[right=28mm of p] (d) {\bf Dual Problems\\ (maximization)};
				
				\node[lp, below=2cm of p.west, anchor=west] (p1) {Probability Measures (\ref{lp:measures})};
				\node[lp, below right=2cm and -2mm of p1] (p2) {Partitions of Unity (\ref{lp:partition})};
				
				\node[lp, below=2cm of d.west, anchor=west] (d1) {Pseudo Flows (\ref{lp:pseudo-flows})};
				\node[lp, below right=2cm and 6mm of d1] (d2) {Isoperimetric Inequalities (\ref{lp:isoperimetric})};
				
				\node[lp, below=80mm of p1] (p5) {Mean Property A (\ref{lp:mean})};
				
				\node[lp, below=35mm of d2] (d4) {Sparsest Cut (Section~\ref{sec:sparsest cut})};
				\node[lp, below=80mm of d1] (d5) {Uniform Flows (\ref{lp:uniform})};
				\node[lp, below=35mm of d4] (d6) {Cheeger Constant (Section~\ref{sec:cheeger})};
				
				\node[above left=1mm and 1mm of d] (sep top) {};
				\node[below left=31mm and 1mm of d5] (sep bot) {};
				
				\path (sep top) edge[dashed, draw=gray] (sep bot);
				
				\path (p5) edge node[align=center] {$=$\\ Duality\\ Theorem~\ref{thm:partition of unity dual}} (d5);
				\path (p1) edge node[align=center] {$\geq$\\ Relaxation} (p5);
				
				\path (p1) edge node[align=center] {$=$\\ Theorem~\ref{thm:measures equals partition}} (p2);
				\path (p1) edge node[align=center] {$=$\\ Duality\\ Theorem~\ref{thm:duality}} (d1);
				\path (d1) edge node[align=center] {$=$\\ Lift and Project\\ Theorem~\ref{thm:lift and project}} (d2);
				
				\path (d1) edge node[align=center] {$\geq$\\ Fix $\eta_i = \eta$\\ Fix $\kappa_{ij} = \frac 1{|E|}$} (d5);
				
				\path (d2) edge node[align=center] {$\geq$\\ Fix $\eta_i = \eta$} (d4);
				\path (d4) edge node[align=center] {$\geq$\\ Fix $\kappa_{ij} = \frac 1{|E|}$} (d6);
				\path (d5) edge node[align=center] {$=$\\ Lift and Project} (d6);
				
				\path (p2) edge node[align=center, near end] {$=$\\ Duality\\ Lemma~\ref{lem:isoperimetric dual to partition}} (d2);
				
			\end{tikzpicture}
		\end{center}
		\caption{The road map of the paper: relationships between most important linear problems and optimization problems.}
		\label{fig: primal and dual problems}
	\end{figure}
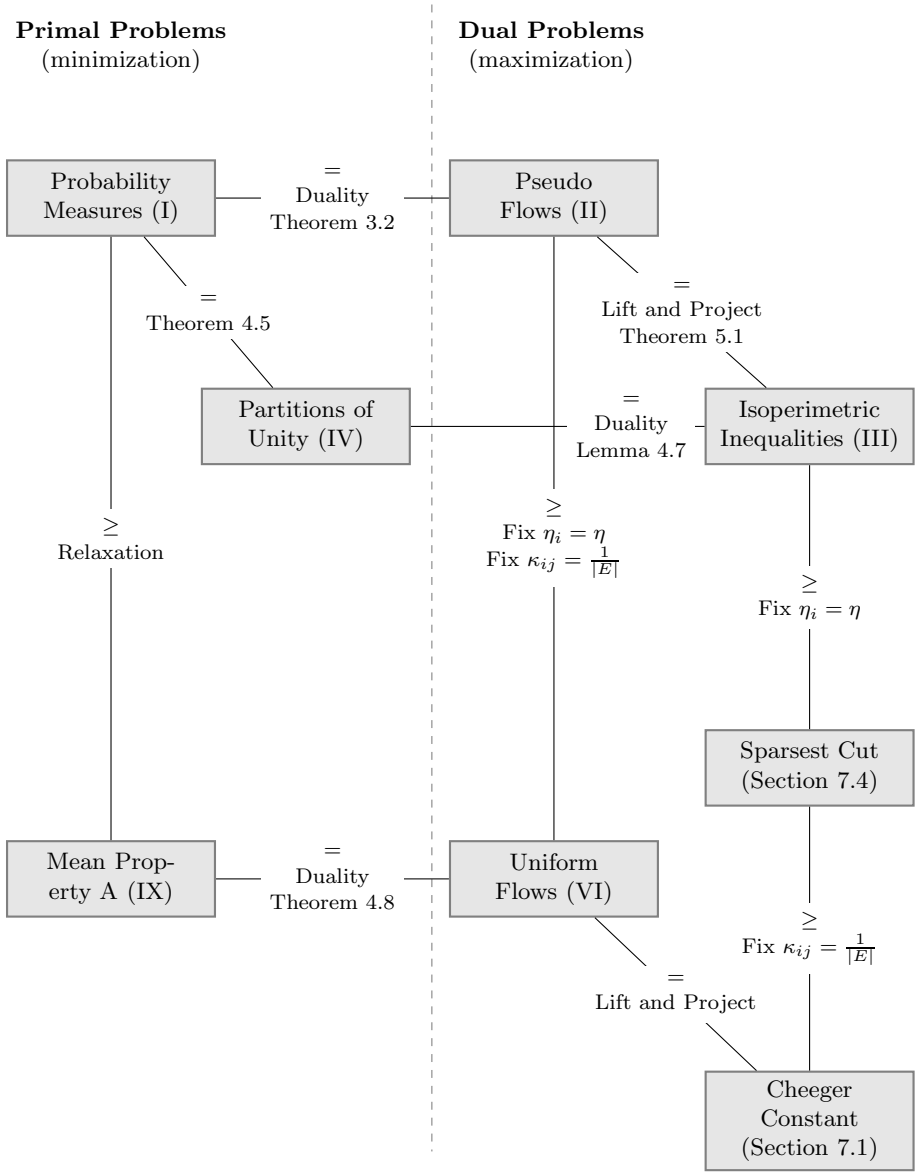
}
\section{Property A as a linear problem}

Our first step is to reduce property~A to a linear programming optimization problem.
While property~A is defined for metric spaces, the most important use case is that of countably infinite graphs with a path-length metric, and we define it in this setting.

We will always assume that our graphs $G = (V, E)$ are oriented ($E \subset V \times V$) with no self-loops.
If $ij \in E$ is used to denote an edge, then $i$ and $j$ are used to denote the initial and terminal vertex of $ij$, respectively.
The distance between vertices the $i$ and $j$ in the path-length metric is the length of a shortest undirected path between $i$ and $j$, which is measured by the number of edges. 
We allow infinite distance.

\begin{definition}
  A \df{scale} on a graph $G = (V,E)$ is a family 
  \[
    \mathcal{S} = \{ S_i \colon S_i \subset V \}_{i \in V}
  \]
  of subsets of $V$, one set per vertex, or an integer $S \geq 0$
  with the associated family
  \[
    \{ S_i = B(i, S) \}_{i \in V},
  \]
  of (closed) balls of radius $S$, centered at the vertices of $G$.
\end{definition}

\begin{op}{Minimal variation of probability measures at scale $\mathcal{S}$}\label{op:measures}
	Let $G = (V, E)$ be a graph and let $\mathcal{S} = \{ S_i \}_{i \in V}$ be a scale on $G$.
	Find the minimal $\varepsilon = \varepsilon_{\mathcal{S}, G}$ (variation) and a family $\{ \xi_i \colon V \to \mathbb{R} \}_{i \in V}$ of functionals (probability measures) such that
	\begin{enumerate}
		\item each $\xi_i$ is a probability measure, i.e.
		\[\| \xi_i \|_1 = 1 \text{ and } \xi_i \geq 0 \text{ for each } i \in V,\]
		\item the variation on each edge $ij$ does not exceed $\varepsilon$, i.e.
		\[
		\| \xi_i - \xi_j \|_1 \leq \varepsilon \text{ for each } ij \in E,
		\]
		\item each $\xi_i$ is supported in $S_i$, i.e.
		\[
		\supp \xi_i = \{ j \in V \colon \xi_i(j) > 0\} \subset S_i \text{ for each } i \in V.
		\]
	\end{enumerate}
\end{op}

Generally the graph $G$ in Optimization Problem~\ref{op:measures} is infinite.

\begin{definition}
  Let $G$ be a graph and for each $S \geq 0$ let $\varepsilon_{S, G}$ be the  minimal variation of probability measures on $G$ at scale $S$, i.e. the solution of Optimization Problem~\ref{op:measures}
  at scale $S$. 
  We say that $G$ has \df{property A} iff
  \[
  \lim_{S \to \infty} \varepsilon_{S, G} = 0.
  \]
\end{definition}

If $S \geq \diam G$, then $\varepsilon_{S, G} = 0$. 
Therefore property~A is trivial for finite graphs.
However, solving Optimization Problem~\ref{op:measures} for finite subgraphs of an infinite graph~$G$ allows us to determine whether or not $G$ has property~A.

\begin{definition}
A graph $G$ is \df{locally finite} if the degree of each vertex of $G$ is finite. 
A subgraph $H$ of a graph $G$ is \df{convex} if for each $i, j \in H$ that lie in a single connected component of $G$ there exists $\gamma$ that is a geodesic from $i$ to $j$ in $G$ such that $\gamma \subset H$.
\end{definition}

For a graph $G$ in the path-length metric, local finiteness is equivalent to the condition that balls are finite while the condition that $H \subset G$ is convex is equivalent to equality of the path-length metrics on $H$ and $G$.

\begin{theorem}\label{thm:property A limit of finite graphs}
  Let $G$ be a countably infinite locally finite graph endowed with the path-length metric.
  Let $G_1 \subset G_2 \subset G_3 \subset \cdots$ be an ascending sequence of 
  finite convex subgraphs of $G$ such that $G = \bigcup_{n \in \mathbb{N}} G_n$.
  Let $\varepsilon_{S, G_n}$ be the minimal variation of probability
  measures at scale $S$ for graph~$G_n$, i.e. the optimal solution of Optimization Problem~\ref{op:measures} for $G_n$ at scale $S$.
  The graph $G$ has property~A iff
  \[
  \lim_{S \to \infty} \limsup_{n \to \infty} \varepsilon_{S, G_n} = 0.
  \]
\end{theorem}

Statements similar to Theorem~\ref{thm:property A limit of finite graphs} were proven by several authors (cf.~\cite{brodzki-niblo-wright2007,willett2009}) although the usual paradigm is to work with infinite graphs. Since it is not germane to the rest of our work, we postpone the proof to the last section.

Presently, we shift focus to finite graphs so we can recast Optimization Problem~\ref{op:measures} as an equivalent linear problem (Linear Problem~\ref{lp:measures}).
This allows us to use duality theory in linear programming to study property~A.

\begin{lp}{Minimal variation of probability measures at scale $S$}\label{lp:measures}
	\begin{mini*}{}{e}{}{}
		\addConstraint{\sum_{j \in V} x_{i, j}=1}{ \text{ for each } i \in V}
		\addConstraint{x_{i,j}=0}{ \text{ for each } i \in V, j \in V\setminus S_i}
		\addConstraint{x_{j,k} - x_{i,k}\leq e_{ij, k}}{ \text{ for each } ij \in E, k \in V}
		\addConstraint{x_{i,k} - x_{j,k}\leq e_{ij, k}}{ \text{ for each } ij \in E, k \in V}
		\addConstraint{\sum_{k \in V} e_{ij, k} \leq e}{ \text{ for each } ij \in E}
		\addConstraint{x_{i, j}, e_{ij, k}, e \geq 0}
	\end{mini*}
\end{lp}

The graph $G = (V, E)$ in Linear Problem~\ref{lp:measures} is assumed to be finite.

\begin{proposition}
  Optimization Problem~\ref{op:measures} and Linear Problem~\ref{lp:measures} are equivalent.
  We can pair admissible solutions of both problems
  by setting 
  \[
    \xi_i(j) = x_{i, j} \text{ and } \varepsilon = e.
  \]
\end{proposition}
\begin{proof}
We have 
\[ 
  \| \xi_i - \xi_j \|_1 = 
  \sum_{k \in V} |\xi_i(k) - \xi_j(k)| =
  \sum_{k \in V} |x_{i, k} - x_{j, k}|
  \leq \sum_{k \in V} e_{ij, k}
  \leq e = \varepsilon.
\]
The other conditions are trivial.
Given an admissible solution of Optimization Problem~\ref{op:measures}, we set $e_{ij, k} = |\xi_i(k) - \xi_j(k)|$ to get a full solution of Linear Problem~\ref{lp:measures}.
\end{proof}

\section{The dual problem}

In this section we construct the dual problem to Linear Problem~\ref{lp:measures}. The dual problem is supported by the dual scale.

\begin{definition}
Let $\mathcal{S} = \{ S_i \}_{i \in V}$ be a scale on a graph $G = (V, E)$.
We let 
\[
\mathcal{\overbar S} = \{ \bar S_i \}_{i \in V} \text{ with } \overbar S_i = \{ j \in V \colon i \in S_j \}
\]
and call $\mathcal{\overbar S}$ the \df{dual scale} to $\mathcal{S}$.
If $\mathcal{S} = \mathcal{\overbar{S}}$ (i.e. $i \in S_j$ iff $j \in S_i$ for each $i, j \in V$), then we call $\mathcal{S}$ \df{symmetric}.
\end{definition}

Note that a scale constructed from balls of radius $S$ is symmetric. Moreover, $i \in \overline{S}_j$ if and only if $j \in S_i$ so the statements below can be rewritten without using the dual scale. We prefer to use it because it provides extra insight.

Linear Problem~\ref{lp:pseudo-flows} at scale $\mathcal{\overbar S}$ 
is dual to Linear Problem~\ref{lp:measures} at scale $\mathcal{S}$ (Theorem~\ref{thm:duality}).

\begin{lp}{Maximal net supply of pseudo-flows at dual scale $\mathcal{\overbar S}$}\label{lp:pseudo-flows}
	\begin{maxi*}[1]
		{}{\sum_{i \in V} \eta_i \hspace{8cm}}{}{\label{lp:dual}}
		\addConstraint{\sum_{ij \in E} \kappa_{ij}}{\leq 1}{}
		\addConstraint{\varphi_{k, ij}}{\leq \kappa_{ij}}{\text{ for each } ij \in E, k \in V}
		\addConstraint{-\varphi_{k, ij}}{\leq \kappa_{ij}}{\text{ for each } ij \in E, k \in V}
		\addConstraint{
			\sum_{mi \in E, m \in V} \varphi_{k, mi} -
			\sum_{im \in E, m \in V} \varphi_{k, im}
		}{ \geq \eta_i}{\text{ for each } k \in V, i \in \overbar S_k}
		\addConstraint{\eta_i, \varphi_{k, ij}}{\in \mathbb{R},}{\ \ \kappa_{ij} \geq 0}
	\end{maxi*}
\end{lp}

Linear Problem~\ref{lp:pseudo-flows} has a natural combinatorial interpretation in terms of pseudo-flows on a network. The network itself is the directed graph $G$ along with a function~$\kappa$ that assigns to each edge $ij\in E$ a non-negative real number, $\kappa_{ij}$, called its \df{capacity}. The total capacity of the network does not exceed $1$: 
\[
\sum_{ij \in E} \kappa_{ij} \leq 1. 
\]

For each vertex $k\in V$ there exists a \df{pseudo-flow} $\varphi_k$ on $E$ such that the magnitude of the flow across the edge $ij$ is bounded by the capacity of that edge: 
\[
|\varphi_{k,ij}|\le \kappa_{ij}.
\]

The quantity
\[
\sigma_{k, i} = 
\sum_{mi \in E, m \in V} \varphi_{k, mi} -
\sum_{im \in E, m \in V} \varphi_{k, im}
\]
is the \df{net supply} of the flow $\varphi_k$ at the vertex $i$.
If $im \in E$ and $\varphi_{k, im} = v$, then the flow $\varphi_k$
  adds the amount $v$ (which may be negative) to $\sigma_{k, m}$ and subtracts $v$ from $\sigma_{k, i}$.

The variable $\eta_i$ is the \df{demand} at node $i \in V$ (which may be negative). The \df{demand constraint} \[
  \sigma_{k, i} = \sum_{mi \in E, m \in V} \varphi_{k, mi} -
	\sum_{im \in E, m \in V} \varphi_{k, im} \geq \eta_i
\]
states that the supply of the flow $\varphi_k$ at node $i$ has to meet the demand. 
It is important to note that the flow $\varphi_k$ has to satisfy these demand constraints \textit{only} for $i \in \overline{S}_k$.

The goal is to maximize the total demand $\sum_{i \in V} \eta_i$ that can be satisfied simultaneously by the pseudo-flows $\varphi_k$ on the subsets $\overline{S}_k$ of the graph.

Let $S, T \subset V$.
We let $E[S, T] = (S \times T \cup T \times S) \cap E$ denote the set of all edges connecting $S$ with $T$ (without regard to orientation).

\begin{op}{Maximal net supply of pseudo-flows at dual scale $\mathcal{\overbar S}$}\label{op:pseudo-flows}
	Let $G = (V, E)$ be a graph and let $\mathcal{\overbar S} = \{ \overbar S_i \}_{i \in V}$ be a dual scale on $G$.
	Find the maximal sum $\sum_{i \in V} \eta_i$ of values of a (demand) functional $\eta \colon V \to \mathbb{R}$
	such that there exists a (capacity) functional $\kappa \colon E \to \mathbb{R}$  and	
	a family $\{ \varphi_k \colon E[\overbar S_k, V] \to \mathbb{R} \}_{k \in V}$ of functionals (pseudo-flows)
	such that 
	\begin{enumerate}
		\item the capacity is non-negative and the total capacity does not exceed $1$, i.e.
		\[ \sum_{ij \in E} \kappa_{ij} \leq 1 \text{ and } \kappa_{ij} \geq 0 \text{ for each } ij \in E, \] 
		\item each flow $\varphi_k$ is bounded by the capacity $\kappa$, i.e.		
		\[
		-\kappa_{ij} \leq \varphi_{k, ij} \leq \kappa_{ij} \text{ for each } ij \in E \text{ and } k \in V,
		\]
		\item the net supply of a flow $\varphi_k$ at node $i \in \overbar S_k$ satisfies the demand $\eta_i$, i.e.
		\[
		\sum_{mi \in E, m \in V} \varphi_{k, mi} - \sum_{im \in E, m \in V} \varphi_{k, im} \geq \eta_i \text{ for each } k \in V, i \in \overbar S_k.
		\]
	\end{enumerate}
\end{op}

\begin{theorem}\label{thm:duality}
	Linear Problem~\ref{lp:pseudo-flows} at scale $\mathcal{\overbar S}$ is dual to Linear Problem~\ref{lp:measures} at scale~$\mathcal{S}$.
	In particular, for each admissible solution of each problem, we have
	\[
	\sum_{i \in V} \eta_i \leq e
	\]
	and the optimal solutions are equal.
\end{theorem}
\begin{proof}
We start with a direct argument (which does not use duality theory) to show the weak duality part of the theorem, i.e. to show the inequality $\sum_{i \in V} \eta_i \leq e$.
The inequalities in the primal Linear Problem~\ref{lp:measures} imply that
\[
e \geq \sum_{k \in V} | x_{j, k} - x_{i, k} |.
\]
Since $\kappa_{ij} \geq |\varphi_{k, ij}|$, for each edge $ij$, we have
\[
\kappa_{ij} \cdot e \geq \sum_{k \in V} \varphi_{k, ij} (x_{j, k} - x_{i, k}).
\]
The condition $\sum_{ij\in E}\kappa_{ij}\le 1$ implies that,
\[
e\ge \sum_{ij \in E} \sum_{k \in V} \varphi_{k, ij} (x_{j, k} - x_{i, k}) =
\sum_{i \in V} \sum_{k \in V} x_{i, k} \left( \sum_{mi \in E} \varphi_{k, mi} - \sum_{im \in E} \varphi_{k, im} \right) 
\geq
\]
\[
\geq \sum_{i \in V} \sum_{k \in V} x_{i,k} \eta_i
= \sum_{i \in V} \eta_i \sum_{k \in V} x_{i,k} = \sum_{i \in V} \eta_i. 
\]
This proves the weak duality part of the theorem.
Note that because $x_{i, k} = 0$ for $i \not \in \overline{S}_k$,
the condition $\sigma_{k, i} \geq \eta_i$ only has to be enforced on nodes $i \in \overline{S}_k$.
Thus a flow $\varphi_k$ can freely draw from nodes that are outside of $\overbar S_k$ and the demand constraints are only defined for those $i \in \overline{S}_k$.

To show strong duality we will construct a dual problem to Linear Problem~\ref{lp:measures} using the method described in~\cite{lahaie2015} (with some shortcuts).
Here $\delta$ denotes the Kronecker delta.
\begin{gather*}
e \geq e 
- \sum_{i \in V} \eta_i (-1 + \sum_{j \in V} x_{i,j}) 
+ \sum_{i \in V} \sum_{j \in V \setminus S_i} \tau_{i,j} x_{i, j}
- \sum_{ij \in E} \sum_{k \in V} \varphi^+_{k, ij} (e_{ij, k} - x_{j,k} + x_{i,k})
\\
\quad - \sum_{ij \in E} \sum_{k \in V} \varphi^-_{k, ij}(e_{ij,k} - x_{i,k} + x_{j,k})
-\sum_{ij \in E} \kappa_{ij} (e - \sum_{k \in V} e_{ij,k}) =
\\
= \sum_{i \in V} \eta_i + e(1 - \sum_{ij \in E} \kappa_{ij}) + 
  \sum_{ij \in E, k \in V} e_{ij, k}(-\varphi^+_{k, ij} - \varphi^-_{k, ij} + \kappa_{ij})
\\
  + \sum_{i, j \in V} x_{i,j}\left(-\sum_{k \in V, ik \in E} \varphi^+_{j, ik} +
    \sum_{k \in V, ki \in E} \varphi^+_{j, ki} 
    - \sum_{k \in V, ki \in E} \varphi^-_{j, ki} \right. \\
    \left. + \sum_{k \in V, ik \in E} \varphi^-_{j, ik}
    + \delta_{j \in V \setminus S_i}\tau_{i,j}  - \eta_i\right) \geq \sum_{i \in V} \eta_i.
\end{gather*}
This gives the dual problem.
\begin{maxi*}[1]
{}{\sum_{i \in V} \eta_i \hspace{8cm}}{}{}
\addConstraint{\sum_{ij \in E} \kappa_{ij}}{\leq 1}{{} }
\addConstraint{\varphi^+_{k, ij} + \varphi^-_{k, ij}}{\leq \kappa_{ij}}{\text{ for each } ij \in E, k \in V}
\addConstraint{
  \sum_{k \in V, ik \in E} (\varphi^-_{j, ik} - \varphi^+_{j, ik}) + 
  \sum_{k \in V, ki \in E} (\varphi^+_{j, ki} - \varphi^-_{j, ki}) + 
  \delta_{j \in V \setminus S_i}\tau_{i,j}}{ \geq \eta_i
}{\text{ for each } i,j \in V}
\addConstraint{\eta_i, \tau_{ij} \in \mathbb{R},\quad \varphi^+_{k, ij}, \varphi^-_{k, ij}, \kappa_{ij}}{ \geq 0}{{} }
\end{maxi*}
Observe that given an admissible solution we can always replace each pair $\varphi^+_{k, ij}, \varphi^-_{k, ij}$ in such a way that either $\varphi^+_{k, ij} = 0$ or $\varphi^-_{k, ij} = 0$.
Then we can substitute $\varphi_{k, ij} = \varphi^+_{k, ij} - \varphi^-_{k, ij}$ and we have $\varphi^+_{k, ij} + \varphi^-_{k, ij} = |\varphi_{k, ij}|$.
Also, the constraints where $\delta_{j \in V\setminus S_i} = 1$ are trivially satisfiable by setting the value of $\tau_{i,j}$ high enough. After these substitutions we obtain Linear Problem~\ref{lp:pseudo-flows}.
\end{proof}

\subsection{Dual problem interpretation}\label{subsec:dual}

We will now delve into an example to give a more detailed explanation of the dual problem.
Consider the following connected chordal graph on $10$ vertices.
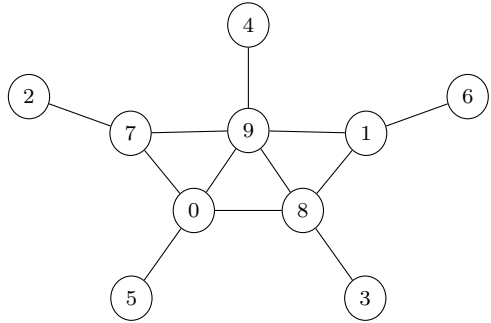
\begin{center}
  \includegraphics[width=0.5\textwidth]{figures/dual/graph.tikz}
\end{center}
This graph was chosen from a set of $109539$ connected chordal graphs on $10$ vertices (all graphs of this type up to isomorphism).
Our interest in chordal graphs will become apparent in Section \ref{sec:girth}.
Among all such graphs, this one has maximal $\varepsilon$ at scale $S=1$; i.e. the minimal variation of probability measures at scale $S = 1$ is maximal for this graph.
This value of $\varepsilon$ is $\frac{16}{17}$, which we will
  demonstrate by exhibiting solutions of the primal and dual linear problems with $\varepsilon = \sum_{i \in V} \eta_i = \frac{16}{17}$.

The following table shows a solution of Linear Problem~\ref{lp:measures}.
The entry in the $i$th column and $j$th row of the table is the value of $x_{i, j}$.
Empty cells denote arguments outside of the scale, i.e. such that $j \not\in S_i$ (implying that $x_{i, j} = 0$).
Values of $e_{ij, k}$ and $e$ are implicit. We can set $e_{ij, k} = |x_{j, k} - x_{i, k}|$ and 
$e = \max_{ij \in E} \left( \sum_{k \in V} e_{ij, k} \right)$ to get an admissible solution of Linear Problem~\ref{lp:measures} with minimal value $e$ for specified values of~$x$.

\renewcommand{\arraystretch}{1.2} 
\begin{longtable}[H]{c|c|c|c|c|c|c|c|c|c|c}
&$\xi_{0}$&$\xi_{1}$&$\xi_{2}$&$\xi_{3}$&$\xi_{4}$&$\xi_{5}$&$\xi_{6}$&$\xi_{7}$&$\xi_{8}$&$\xi_{9}$\\ \midrule
0 & $\frac{9}{17}$ &       &       &       &       & $1$ &       & $\frac{3}{17}$ & $\frac{2}{17}$ & $\frac{1}{17}$\\
1 &       & $\frac{1}{17}$ &       &       &       &       & $\frac{9}{17}$ &       & $\frac{1}{17}$ & $\frac{1}{17}$\\
2 &       &       & $\frac{14}{17}$ &       &       &       &       & $\frac{6}{17}$ &       &      \\
3 &       &       &       & $\frac{6}{17}$ &       &       &       &       & $\frac{6}{17}$ &      \\
4 &       &       &       &       & $\frac{12}{17}$ &       &       &       &       & $\frac{4}{17}$\\
5 & $0$ &       &       &       &       & $0$ &       &       &       &      \\
6 &       & $\frac{8}{17}$ &       &       &       &       & $\frac{8}{17}$ &       &       &      \\
7 & $\frac{1}{17}$ &       & $\frac{3}{17}$ &       &       &       &       & $\frac{3}{17}$ &       & $\frac{3}{17}$\\
8 & $\frac{2}{17}$ & $\frac{3}{17}$ &       & $\frac{11}{17}$ &       &       &       &       & $\frac{3}{17}$ & $\frac{3}{17}$\\
9 & $\frac{5}{17}$ & $\frac{5}{17}$ &       &       & $\frac{5}{17}$ &       &       & $\frac{5}{17}$ & $\frac{5}{17}$ & $\frac{5}{17}$\\
\bottomrule\end{longtable}

As before we let $\xi_i(j) = x_{i, j}$, 
so that each column describes a functional $\xi_i \colon S_i \to \mathbb{R}$ from Optimization Problem~\ref{op:measures}.
Clearly, $\xi_i \geq 0$, $\| \xi_i \|_1 = 1$ (condition (1)), and $\supp \xi_i \subset B(i, 1)$ (condition (3)). To check condition (2) let
\[
e_{ij} = \sum_{k \in V} e_{ij, k} = \| \xi_i - \xi_j \|_1
\]
be an auxiliary variable that denotes the variation on the edge $ij$.
For the solution given in the table above the values of $e_{ij}$ are shown on the following graph.
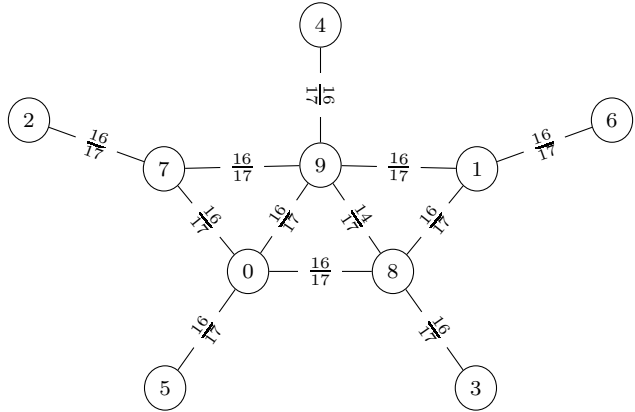
\begin{center}
  \includegraphics[width=0.65\textwidth]{figures/dual/variation.tikz}
\end{center}
We see that $e = \frac{16}{17}$ is realized on all edges except the edge between nodes $8$ and $9$. The solution is admissible in the sense that it satisfies all constraints. 
It was found using the simplex method and it is one out of many basic optimal solutions of the problem (we did not show optimality yet).

To show that this solution is optimal we will provide a solution of the dual Linear Problem~\ref{lp:pseudo-flows} with the same value of the objective function.
The values of the variables $\kappa_{ij}$ and $\eta_i$ of the solution are shown on the following graph. The values of $\kappa$ are shown as edge labels and the values of $\eta$ are shown as node labels.

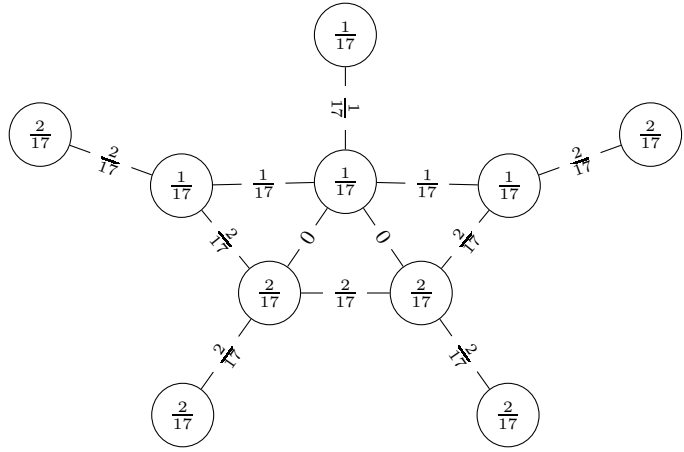
\begin{center}
  \includegraphics[width=0.70	 \textwidth]{figures/dual/capacity.tikz}
\end{center}

We interpret the edge labels $\kappa_{ij}$ as a capacity function $\kappa \colon E \to [0, \infty)$ and the node labels $\eta_i$ as a demand function $\eta \colon V \to \mathbb{R}$.
We have $\sum_{ij \in E} \kappa_{ij} = 1$ (total capacity) and $\sum_{i \in V} \eta_i = \frac{16}{17}$ (total demand).
To complete the solution of Linear Problem~\ref{lp:pseudo-flows} we only need to specify the values of $\varphi_{k, ij}$. The main point is that once we set $\kappa$ and $\eta$ this task can be done separately for each~$k$.
For fixed $\kappa$, $\eta$ and $k$ we describe the process as a pseudo-flow optimization problem, as follows.

Fix $k = 0$. The following graph shows the values of $\varphi_{k, ij}$ for each $ij \in E[\overbar S_k, V]$. So far we have ignored the orientation of the graph; now we fix it to show the direction of the flow (where a negative value means flow in the direction opposite to the edge's orientation) and the flow amount is shown on each edge. 

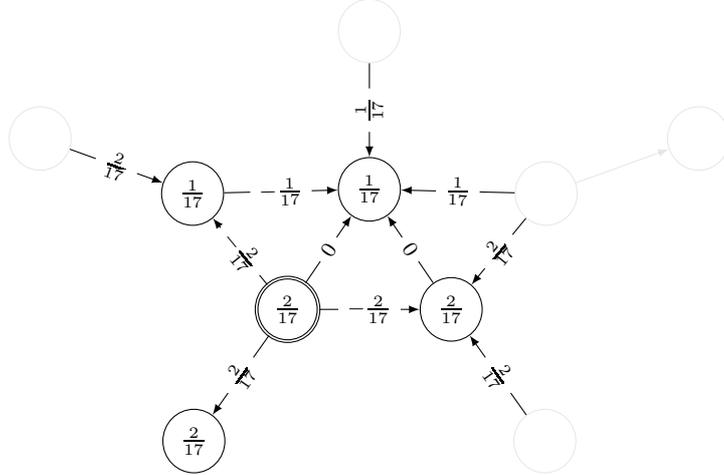
\begin{center}
	\includegraphics[width=0.75\textwidth]{figures/dual/flow_0.tikz}
\end{center}

The fixed node $k = 0$ is marked with a double boundary.
The values of $\varphi_{k, ij}$ (for $k=0$) are indicated on the edges.
One can see that the flow $\varphi_{k, ij}$ does not exceed the capacity $\kappa_{ij}$ by directly comparing the values in the graphs.
We interpret the values of $\varphi_{k, ij}$ as a function $\varphi_k \colon E[\overbar S_k, V] \to \mathbb{R}$---a pseudo-flow---with supply
\[
\sigma_{k, i} = \sum_{mi \in E, m \in V} \varphi_{k, mi} -
\sum_{im \in E, m \in V} \varphi_{k, im}
\]
at node $i \in \overbar S_k$.
Condition (3) of Linear Problem~\ref{lp:pseudo-flows} states that 
$ \sigma_{k, i} \geq \eta_i $, i.e.
that the supply $\sigma_{k, i}$ at node $i \in \overbar S_k$ of the flow $\varphi_k$ meets the demand $\eta_i$. To verify this, we mark the values of $\sigma_{k, i}$  as node labels in the graph and compare the values on both graphs directly. We see that
  the demand is satisfied by the supply at each node.

For $k = 6$ the pseudo-flow is the following.
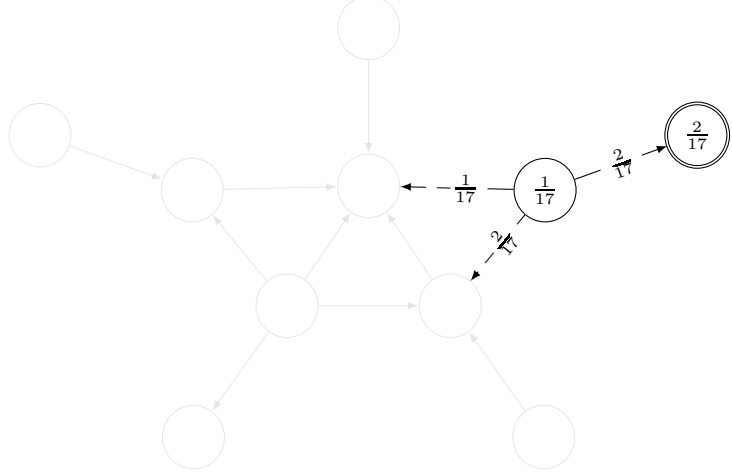
\begin{center}
	\includegraphics[width=0.75\textwidth]{figures/dual/flow_6.tikz}
\end{center}
Again, the flow does not exceed the capacity $\kappa$ and the demand $\eta$ is satisfied by the flow.
Pseudo-flows for the other vertices of the graph can be constructed as well.
For brevity, we skip the details.

The following table shows the supply of each flow at each node.
Note that we construct separate flows for each column, but the 
  demand that can be satisfied is determined in each row, which makes interactions between
  $\varphi$ and $\eta$ subtle.
  
\begin{longtable}[H]{c|c|c|c|c|c|c|c|c|c|c|c}
i&$\sigma_{0, i}$&$\sigma_{1, i}$&$\sigma_{2, i}$&$\sigma_{3, i}$&$\sigma_{4, i}$&$\sigma_{5, i}$&$\sigma_{6, i}$&$\sigma_{7, i}$&$\sigma_{8, i}$&$\sigma_{9, i}$& $\min_k \sigma_{k, i}$\\ \midrule
$0$ & $\frac{2}{17}$ &       &       &       &       & $\frac{2}{17}$ &       & $\frac{2}{17}$ & $\frac{2}{17}$ & $\frac{2}{17}$& $\eta_0 = $ $\frac{2}{17}$\\
$1$ &       & $\frac{1}{17}$ &       &       &       &       & $\frac{1}{17}$ &       & $\frac{1}{17}$ & $\frac{1}{17}$& $\eta_1 = $ $\frac{1}{17}$\\
$2$ &       &       & $\frac{2}{17}$ &       &       &       &       & $\frac{2}{17}$ &       &      & $\eta_2 = $ $\frac{2}{17}$\\
$3$ &       &       &       & $\frac{2}{17}$ &       &       &       &       & $\frac{2}{17}$ &      & $\eta_3 = $ $\frac{2}{17}$\\
$4$ &       &       &       &       & $\frac{1}{17}$ &       &       &       &       & $\frac{1}{17}$& $\eta_4 = $ $\frac{1}{17}$\\
$5$ & $\frac{2}{17}$ &       &       &       &       & $\frac{2}{17}$ &       &       &       &      & $\eta_5 = $ $\frac{2}{17}$\\
$6$ &       & $\frac{2}{17}$ &       &       &       &       & $\frac{2}{17}$ &       &       &      & $\eta_6 = $ $\frac{2}{17}$\\
$7$ & $\frac{1}{17}$ &       & $\frac{1}{17}$ &       &       &       &       & $\frac{1}{17}$ &       & $\frac{1}{17}$& $\eta_7 = $ $\frac{1}{17}$\\
$8$ & $\frac{2}{17}$ & $\frac{2}{17}$ &       & $\frac{2}{17}$ &       &       &       &       & $\frac{2}{17}$ & $\frac{2}{17}$& $\eta_8 = $ $\frac{2}{17}$\\
$9$ & $\frac{1}{17}$ & $\frac{1}{17}$ &       &       & $\frac{1}{17}$ &       &       & $\frac{1}{17}$ & $\frac{1}{17}$ & $\frac{1}{17}$& $\eta_9 = $ $\frac{1}{17}$\\
\bottomrule & & & & & & & & & & & $ \sum = $ $\frac{16}{17}$ \\
\end{longtable}

\label{supply table}

We specified a solution of Linear Problem~\ref{lp:pseudo-flows} such that the value of objective function for this flow is $\sum_{i \in V} \eta_i = \frac{16}{17}$. By Theorem~\ref{thm:duality} this solution is optimal.

\subsection{The square graph \texorpdfstring{$\mathbb{Z}_2 \times \mathbb{Z}_2$}{}}

We will show what each pseudo-flow $\varphi_k$ is transferring on a minimal non-trivial example. Consider the following graph.
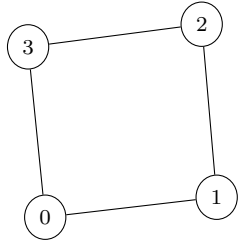
\begin{center}
  \includegraphics[width=0.24\textwidth]{figures/Cr_1/graph.tikz}
\end{center}
It is not difficult to find a good candidate for the optimal solution to Optimization Problem~\ref{op:measures} at scale $S=1$, by setting $\xi_i(j)=\frac13$ when $j\in S_i = B(i, 1)$ and $0$ otherwise. This yields the value $\varepsilon = \frac23$.

For the dual problem, we set each capacity $\kappa_{ij} = \frac 14$ and construct flows that satisfy the demand $\eta_i = \frac 16$ at each node.
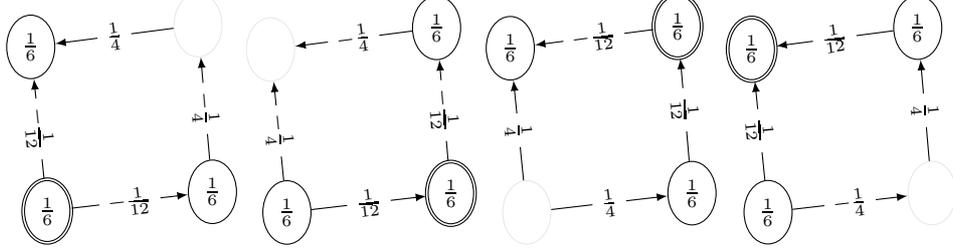
\begin{center}
  \includegraphics[width=0.24\textwidth]{figures/Cr_1/uniform_flow_0.tikz}
  \includegraphics[width=0.24\textwidth]{figures/Cr_1/uniform_flow_1.tikz}
  \includegraphics[width=0.24\textwidth]{figures/Cr_1/uniform_flow_3.tikz}
  \includegraphics[width=0.24\textwidth]{figures/Cr_1/uniform_flow_2.tikz}
\end{center}
The diagram above indicates the values of the pseudo-flows $\varphi_0, \varphi_1, \varphi_2, \varphi_3$, from left to right, with the node $i$ of the flow $\varphi_i$ indicated by a double boundary. The values on the edges are the pseudo-flow values. The labels on each vertex indicate the supplies of the flow; we note that these are marked only in nodes from the dual scale, since the demand constraints are not defined outside the dual scale.

We will show how the dual solution proves the optimality of the primal solution.
The following inequality describes the variation of the probability measures on edge~$01$.
\[
\| \xi_0 - \xi_1 \|_1 = 
|x_{1, 0}- x_{0,0}|+
|x_{1, 1} - x_{0,1}|+
|x_{1, 2}- {\color{gray!50}x_{0,2}}|+
|{\color{gray!50}x_{1,3}} - x_{0,3}| \leq \varepsilon
\]
Those values that are not in the support, and are therefore forced to be $0$ by the problem specification, are grayed out.
This implies that
\[
-\frac{1}{3} \cdot \left(x_{1,0} - x_{0,0}\right) + 
\frac{1}{3} \cdot \left(x_{1,1} - x_{0,1}\right) + 
1 \cdot (\left(x_{1,2} - {\color{gray!50}x_{0,2}}\right) +
(-1) \cdot \left({\color{gray!50}x_{1,3}} - x_{0,3}\right) \leq
\varepsilon.
\]
This inequality remains valid if we replace the coefficients with any coefficients from the interval $[-1, 1]$; the reason for our choice of coefficients will be apparent below: we use rescaled values of the first flow shown above.
We can write similar inequalities using all flows
to obtain the following system of inequalities; here we have multiplied each inequality by $\frac14$ so that the right-hand side sums to $\varepsilon$.

\begin{gather*}
-\frac{1}{12} \left(x_{1,0} - x_{0,0}\right) + 
\frac{1}{12} \left(x_{1,1} - x_{0,1}\right) + 
\frac 14 \left(x_{1,2} - {\color{gray!50}x_{0,2}}\right) +
(-\frac 14) \left({\color{gray!50}x_{1,3}} - x_{0,3}\right) \leq
\frac 14 \varepsilon
\\
-\frac{1}{12} \left(x_{3, 0} - x_{0, 0}\right) + 
(-\frac{1}{4}) \left({\color{gray!50} x_{3, 1}} - x_{0, 1}\right) +
\frac{1}{4} \left(x_{3, 2} - {\color{gray!50}x_{0, 2}}\right) + 
\frac{1}{12} \left(x_{3, 3} - x_{0, 3}\right) \leq 
\frac{1}{4} \varepsilon
\\
-\frac{1}{4} \left({\color{gray!50}x_{2, 0}} - x_{1, 0}\right) +
(-\frac{1}{12}) \left(x_{2, 1} - x_{1, 1}\right) + 
\frac{1}{12} \left(x_{2,2} - x_{1, 2}\right) + 
\frac{1}{4} \left(x_{2, 3} - {\color{gray!50}x_{1, 3}}\right) \leq 
\frac{1}{4} \varepsilon
\\
\frac{1}{4} \left(x_{3, 0} - {\color{gray!50}x_{2, 0}}\right) + 
(-\frac{1}{4}) \left({\color{gray!50}x_{3, 1}} - x_{2, 1}\right) + 
(-\frac{1}{12}) \left(x_{3, 2} - x_{2, 2}\right) + 
\frac{1}{12} \left(x_{3, 3} - x_{2, 3}\right) \leq 
\frac{1}{4} \varepsilon
\end{gather*}

The coefficients of this system of linear equations in matrix notation are
{
\renewcommand{\arraystretch}{1.2}
\[
\left[
\begin{array}{rrrr|r}
-\frac 1{12} & \frac 1{12} & \frac 14 & -\frac 14 & \frac 14 \\
-\frac 1{12} & -\frac 14 & \frac 14 & \frac 1{12} & \frac 14 \\
-\frac 14 & -\frac 1{12} & \frac 1{12} & \frac 14 & \frac 14 \\
\frac 14 & -\frac 14 & -\frac 1{12} & \frac 1{12} & \frac 14 
\end{array}\right].
\]
}

The coefficient matrix has $5$ columns and $4$ rows.
Each row corresponds to an edge of the graph: $01$, $03$, $12$, $23$ respectively.
For the row corresponding to edge $ij \in E$, the coefficient on the right-hand side is $\kappa_{ij}$. 
Each column on the left-hand side corresponds to one vertex of the graph: $0$, $1$, $2$, $3$ respectively. 
For the column corresponding to vertex $k \in V$, the coefficients are the values of the flow $\varphi_k$.

We add everything together and group by $x$'s.

\begin{gather*}
x_{0, 0} (-(-\frac{1}{12})-(-\frac{1}{12})) +
x_{0, 1} (-\frac{1}{12}-(-\frac{1}{4})) + 
{\color{gray!50}x_{0, 2} (-\frac{1}{4}-\frac{1}{4})} + 
x_{0, 3} (-(-\frac{1}{4})-\frac{1}{12}) +
\\
x_{1, 0} (-\frac{1}{12}-(-\frac{1}{4})) + 
x_{1, 1} (\frac{1}{12}-(-\frac{1}{12})) + 
x_{1, 2} (\frac{1}{4}-\frac{1}{12}) + 
{\color{gray!50}x_{1, 3} (-\frac{1}{4}-\frac{1}{4})}
\\
{\color{gray!50}x_{2, 0}(-\frac{1}{4}-\frac{1}{4})} +
x_{2,1} (-\frac{1}{12}-(-\frac{1}{4})) + 
x_{2,2} (\frac{1}{12}-(-\frac{1}{12})) + 
x_{2,3} (\frac{1}{4}-\frac{1}{12})
\\
x_{3,0} (-\frac{1}{12}+\frac{1}{4}) +
{\color{gray!50}x_{3, 1} (-\frac{1}{4}+-\frac{1}{4})} +
x_{3,2} (\frac{1}{4}+-\frac{1}{12}) + 
x_{3,3} (\frac{1}{12}+\frac{1}{12})
 \leq \varepsilon
\end{gather*}

After regrouping we get a new coefficient matrix:
{
\renewcommand{\arraystretch}{1.2}
\[
\begin{bmatrix}
-(-\frac{1}{12})-(-\frac{1}{12}) &
-\frac{1}{12}-(-\frac{1}{4}) &
 &
-(-\frac{1}{4})-\frac{1}{12}
\\
-\frac{1}{12}-(-\frac{1}{4}) & 
\frac{1}{12}-(-\frac{1}{12}) & 
\frac{1}{4}-\frac{1}{12}) & 

\\
 &
-\frac{1}{12}-(-\frac{1}{4}) &
\frac{1}{12}-(-\frac{1}{12}) &
\frac{1}{4}-\frac{1}{12}
\\
-\frac{1}{12}+\frac{1}{4} &
 &
\frac{1}{4}+-\frac{1}{12} & 
\frac{1}{12}+\frac{1}{12}
\end{bmatrix}.
\]
}

The coefficients in the first row are $\sigma_{0, 0}, \sigma_{1, 0}, \sigma_{3, 0}$: the supplies of the flows $\varphi_0, \varphi_1, \varphi_3$ at node~$0$. (The value $\sigma_{2, 0}$ is ignored since $2$ is not in the support of $\xi_0$ and $x_{0,2} = 0$.)
By the demand constraints we have $\sigma_{0, 0}, \sigma_{1, 0}, \sigma_{3, 0} \geq \eta_0 = \frac 16$ and by the probability measure constraint we have $x_{0,0} + x_{0,1} + x_{0,3} = 1$ (again, $x_{0,2} = 0$).
Therefore,
\[
  \sigma_{0,0} x_{0,0} + \sigma_{1, 0} x_{0,1} + \sigma_{2,0} x_{0,2} \geq \eta_0 (x_{0,0} + x_{0, 1} + x_{0,2}) = \frac 16 \cdot 1.
\]
A similar argument for subsequent rows gives a lower bound of $4 \cdot \frac 16$ on $\varepsilon$, which proves that the primal solution $\varepsilon = \frac 23$ was optimal.

The game is to maximize the minimum value $\eta_i$ in each row of the coefficient matrix in such a way that $\sum_{i \in V} \eta_i$ is maximal. A change in a single flow value $\varphi_{k, ij}$ by $v$ increases the value of $\sigma_{k, j}$ by $v$ (in the $j$th row, $k$th column) and decreases the value of $\sigma_{k, i}$ by $v$.
The flow $\varphi_k$ along the edge $ij$ transfers coefficient from $k$th column, $i$th row to $k$th column, $j$th row.
The values that are not in the scale are ignored, which allows us to obtain a positive net value.

The relationship between supplies and demands in the coefficient matrix is rather subtle: a flow $\varphi_k$ determines the supplies in the $k$th column, while a demand $\eta_i$ is computed from the $i$th row.
A natural simplification when constructing a flow $\varphi_k$ is to maximize its minimal supply (i.e. maximize the minimal coefficient in $k$th column), which splits the problem into several easier sub-problems.
Later, we will achieve this by replacing the possibly distinct demands $\eta_i$ by a single global demand~$\eta$.
The resulting problem is easier, but not equivalent. 
We will study this in detail in Section~\ref{sec:sparsest cut}.

\section{The dual problem without flows}

In this section we give a reformulation of the dual problem that does not use the flow variables~$\varphi_{k, ij}$.
We show that its dual (which is again a minimization problem) is a well-known reformulation of property A in the language of partitions of unity.

\subsection{The cube graph \texorpdfstring{$\mathbb{Z}_2 \times \mathbb{Z}_2 \times \mathbb{Z}_2$}{}} To motivate what follows, we solve the primal and dual problems for the three-dimensional cube graph $Q_3$ at scale $S=2$, see Figure~\ref{fig:cube}.

\begin{figure}
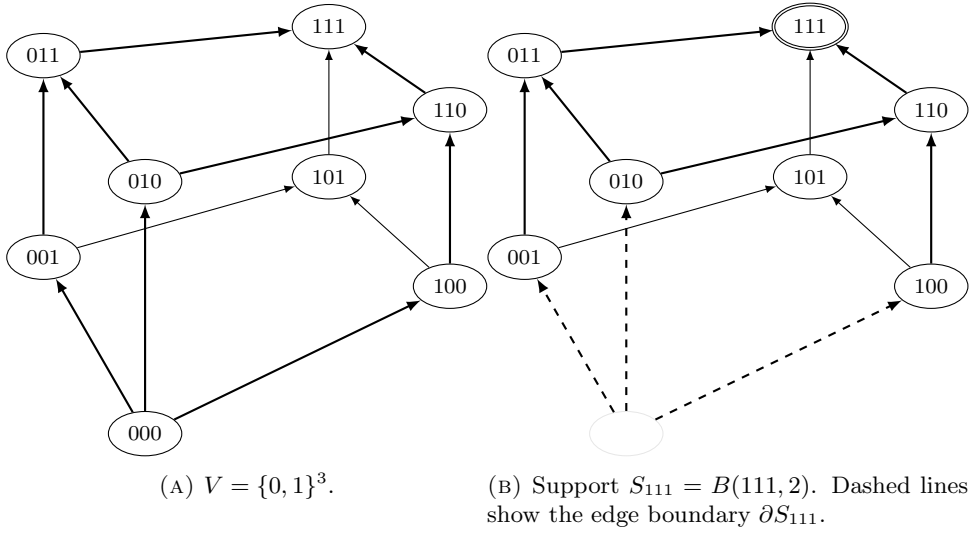

\centering
\begin{subfigure}[t]{0.5\textwidth}
\includegraphics[width=\textwidth]{figures/GS/graph.tikz}%
\caption{$V = \{ 0, 1 \}^3$.}
\end{subfigure}%
\begin{subfigure}[t]{0.5\textwidth}
\includegraphics[width=\textwidth]{figures/GS/support_111.tikz}%
\caption{Support $S_{111} = B(111, 2)$. Dashed lines show the edge boundary $\partial S_{111}$.}
\label{fig:cube support}
\end{subfigure}
\caption{The cube graph $Q_3$ at scale $S=2$.}
\label{fig:cube}
\end{figure}
For $S = 2$ the scale is $\mathcal{S}=\{ S_i = B(i, 2) \colon i\in V \}$.
It is not difficult achieve a value of $e=\frac27$ for the primal problem by setting $x_{i,j}=\frac17$ for each $i\in V$ when $j\in S_i$ and $0$ otherwise. 

To solve the dual problem, we set the capacity of each edge to be $\kappa_{ij} = \frac{1}{|E|} = \frac{1}{12}$.
Because of the symmetry of the graph, it is reasonable to maximize the minimal supply $\eta = \min_{i \in V} \eta_i$ over all nodes (cf. Theorem~\ref{thm:averaging}).
Our task is to find, for each $k$, a pseudo-flow on $E[S_k, V]$ that maximizes $\eta$.
The edge set $E[S_k, V]$ for $k = 111$ is shown on Figure~\ref{fig:cube support}.
The dashed lines are edges at the boundary $\partial S_k$.
Observe that the flow on the edges connecting vertices of $S_k$ with vertices of $S_k$ does not change the total supply $\sum_{i \in S_k} \sigma_{k, i}$ of the flow $\varphi_k$ on $S_k$. 
The total supply is equal to the amount of flow on $\partial S_k$, which is bounded by $|\partial S_k| \cdot \frac 1{|E|} = \frac 3{12}$.
Clearly, the minimal supply $\eta$ at a node is not greater than the total supply divided by the number of nodes in $S_k$:
\[
\eta \leq \frac {|\partial S_k|}{|S_k| \cdot |E|} = \frac 3{12 \cdot 7}.
\]
Here it is easy to find a flow that spreads the supply evenly across the nodes of $S_k$ and obtain $\eta_i = \frac 1{28}$ for each i.
For $k=111$, we flow the maximal value of $\frac1{12}$ along the edges from $000$ to each of $001, 010$, and $100$. Next, we flow $\frac1{42}$ from each of these three vertices to the pair of vertices above them (whose descriptions contain exactly two $1$s). Finally, we flow $\frac1{84}$ from each vertex containing exactly two $1$s to the vertex $111$. This yields a difference $\sigma$ of $\frac{3}{84}=\frac{1}{28}$ at each vertex. 

By symmetry, we can obtain the same value for each starting vertex; thus $\eta_i=\frac 1{28}$ and $\sum_{i\in V}\eta_i=8\cdot\frac1{28}=\frac27$.
By Theorem~\ref{thm:duality}, we have
\[\frac27\le\sum_{i\in V}\eta_i\le e\le\frac27.\]

\begin{figure}
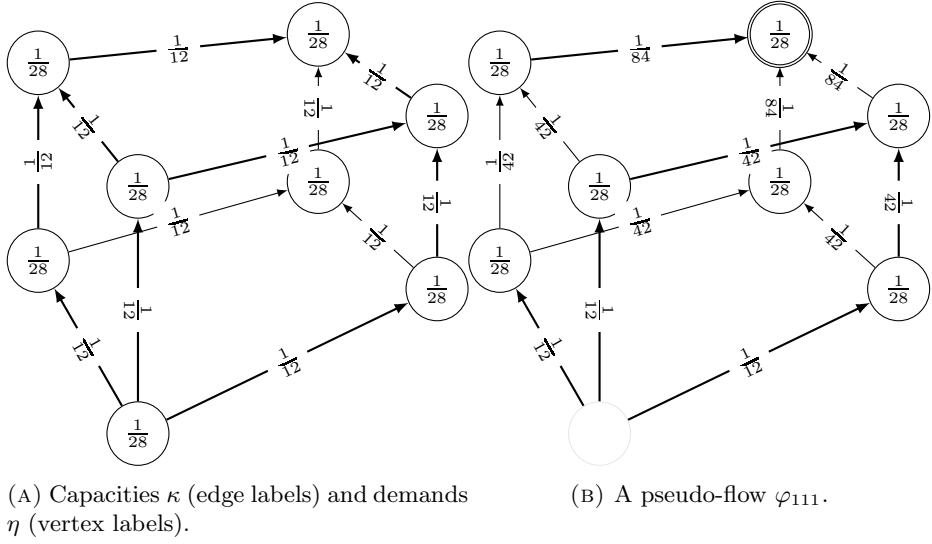

\centering
\begin{subfigure}[t]{0.48\textwidth}
\includegraphics[width=\textwidth]{figures/GS/reoptimized_capacity.tikz}%
\caption{Capacities $\kappa$ (edge labels) and demands $\eta$ (vertex labels).}
\end{subfigure}%
\begin{subfigure}[t]{0.48\textwidth}
\includegraphics[width=\textwidth]{figures/GS/uniform_flow_111.tikz}
\caption{A pseudo-flow $\varphi_{111}$.}
\end{subfigure}
\caption{A symmetric solution of Optimization Problem~\ref{op:pseudo-flows} for $Q_3$ at scale $S=2$.}
\end{figure}

Unsurprisingly, this manual solution is only one of the possible pseudo-flows that achieve this optimal value of $\sum\eta_i$. A basic solution found by the simplex method lacks the symmetry and is described in Figure~\ref{fig:cube alternate flow}, which is part of a solution with $\sum\eta_i=\frac27$; however, the manual solution yields insight that leads to a general statement, see Theorem~\ref{thm:cube}.

\begin{figure}[H]
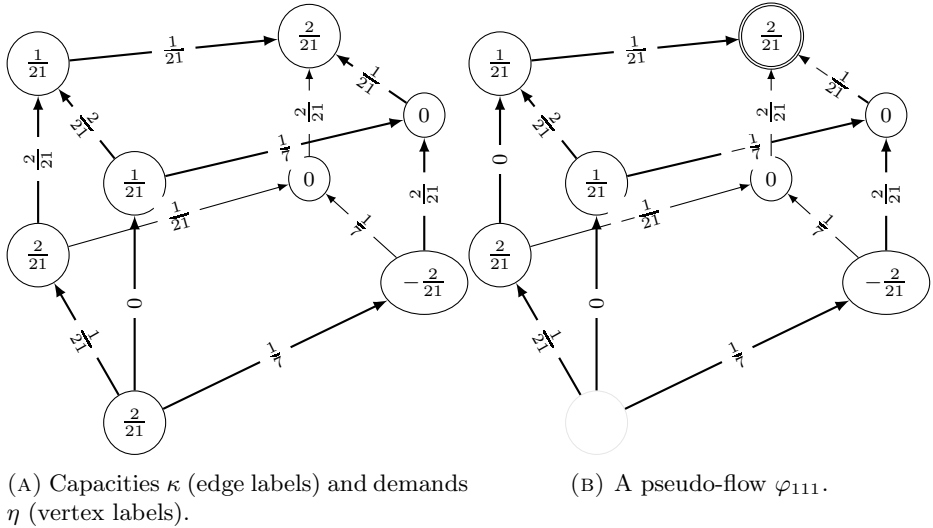

\centering
\begin{subfigure}[t]{0.48\textwidth}
\includegraphics[width=\textwidth]{figures/GS/capacity.tikz}%
\caption{Capacities $\kappa$ (edge labels) and demands $\eta$ (vertex labels).}
\end{subfigure}%
\begin{subfigure}[t]{0.48\textwidth}
\includegraphics[width=\textwidth]{figures/GS/flow_111.tikz}
\caption{A pseudo-flow $\varphi_{111}$.}
\end{subfigure}
\caption{A basic solution (one of many) of Optimization Problem~\ref{op:pseudo-flows} for $Q_3$ at scale $S=2$ found with the simplex method. Note negative value $\eta_{100} = - \frac{2}{21}$.}
\label{fig:cube alternate flow}
\end{figure}

\subsection{The \texorpdfstring{$3\times 3$}{3 by 3} grid} \label{ex:3x3} The optimal solutions to the primal and dual problems in the previous example could be found by hand by exploiting the cube's inherent symmetry. Next, we consider a $3\times 3$ grid.%, which has three different types of vertices, exemplified by $0$, $1$, and $4$ in Figure~\ref{fig:3x3 grid}. 

%\begin{figure}[H]
%\begin{center}
%	\input{figures/Hk/graph.tikz}
%\end{center}
%\caption{The $3\times 3$ grid graph does not enjoy the symmetry of the cube.}
%\label{fig:3x3 grid}
%\end{figure}

Fix $S=1$. Attempting to find a solution to the dual problem with equal edge capacities of $\kappa=\frac{1}{|E|}=\frac1{12}$ using the technique that we employed above results in a ``bottleneck'' along an edge incident to the center vertex as indicated below. 
%in Figure~\ref{fig:flow at 1}.  
%\begin{figure}[H]
    
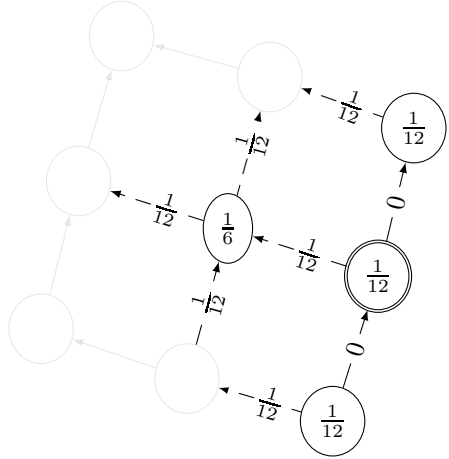
\begin{center}
    \begin{tikzpicture}[scale=.35,>=latex,line join=bevel]
\node (0) at (414.6bp,17.93bp) [draw,ellipse] {$\frac{1}{12}$};
  \node (1) at (463.02bp,173.6bp) [draw,ellipse,double] {$\frac{1}{12}$};
  \node (3) at (259.5bp,63.732bp) [draw,ellipse,opacity=0.1] {\phantom{$\frac{1}{12}$}};
  \node (2) at (501.07bp,332.73bp) [draw,ellipse] {$\frac{1}{12}$};
  \node (4) at (303.08bp,224.92bp) [draw,ellipse] {$\frac{1}{6}$};
  \node (5) at (347.77bp,387.51bp) [draw,ellipse,opacity=0.1] {\phantom{$\frac{1}{12}$}};
  \node (6) at (104.12bp,117.2bp) [draw,ellipse,opacity=0.1] {\phantom{$\frac{1}{12}$}};
  \node (7) at (144.06bp,275.87bp) [draw,ellipse,opacity=0.1] {\phantom{$\frac{1}{12}$}};
  \node (8) at (189.77bp,431.51bp) [draw,ellipse,opacity=0.1] {\phantom{$\frac{1}{12}$}};
  \draw [->] (0) -- node[fill=white, sloped] {$0$} (1);
  \draw [->] (0) -- node[fill=white, sloped] {$-\frac{1}{12}$} (3);
  \draw [->] (1) -- node[fill=white, sloped] {$0$} (2);
  \draw [->] (1) -- node[fill=white, sloped] {$-\frac{1}{12}$} (4);
  \draw [->] (2) -- node[fill=white, sloped] {$-\frac{1}{12}$} (5);
  \draw [->] (3) -- node[fill=white, sloped] {$\frac{1}{12}$} (4);
  \draw [->,opacity=0.1] (3) -- (6);
  \draw [->] (4) -- node[fill=white, sloped] {$-\frac{1}{12}$} (5);
  \draw [->] (4) -- node[fill=white, sloped] {$-\frac{1}{12}$} (7);
  \draw [->,opacity=0.1] (5) -- (8);
  \draw [->,opacity=0.1] (6) -- (7);
  \draw [->,opacity=0.1] (7) -- (8);
\end{tikzpicture}

    \end{center}
%    \caption{With equal demands $\kappa=\frac{1}{12}$, we cannot flow enough charge through the edge $14$ to achieve equal demands at each vertex; excess charge is left at $4$.}
%    \label{fig:flow at 1}
%\end{figure}

This solution results in $\sum\eta_i=\frac{41}{45}$. On the other hand, by allowing variable capacities, one can find the optimal value $\sum_i\eta_i=\frac{12}{13}$, see Figure~\ref{fig:3x3_optimal_dual}. 

%\input{figures/HkSg_SD_1/table.tex}

%Indeed, this value of $\frac{12}{13}$ is optimal and can be achieved by pseudo-flows as follows with capacities on edges and demand as indicated in Figure~\ref{fig:3x3_optimal_dual}.
\begin{figure}[H]
    \centering
    \begin{tikzpicture}[scale=.4,>=latex,line join=bevel,font=\footnotesize]
\node (0) at (295.64bp,17.547bp) [draw,ellipse] {$0$};
  \node (1) at (337.14bp,150.95bp) [draw,ellipse] {$\frac{1}{13}$};
  \node (3) at (162.72bp,56.798bp) [draw,ellipse] {$\frac{3}{13}$};
  \node (2) at (369.74bp,287.32bp) [draw,ellipse] {$\frac{2}{13}$};
  \node (4) at (200.07bp,194.93bp) [draw,ellipse] {$\frac{2}{13}$};
  \node (5) at (238.36bp,334.27bp) [draw,ellipse] {$\frac{1}{13}$};
  \node (6) at (29.563bp,102.62bp) [draw,ellipse] {$0$};
  \node (7) at (63.796bp,238.59bp) [draw,ellipse] {$\frac{2}{13}$};
  \node (8) at (102.96bp,371.98bp) [draw,ellipse] {$\frac{1}{13}$};
  \draw [] (0) ..controls (309.04bp,60.636bp) and (323.49bp,107.09bp)  .. node[fill=white, sloped] {$\frac{1}{13}$} (1);
  \draw [] (0) ..controls (253.01bp,30.136bp) and (228.23bp,37.452bp)  .. node[fill=white, sloped] {$\frac{1}{13}$} (3);
  \draw [] (1) ..controls (347.7bp,195.14bp) and (359.15bp,243.01bp)  .. node[fill=white, sloped] {$\frac{1}{13}$} (2);
  \draw [] (1) ..controls (278.22bp,169.86bp) and (259.09bp,175.99bp)  .. node[fill=white, sloped] {$\frac{1}{13}$} (4);
  \draw [] (2) ..controls (313.51bp,307.41bp) and (294.45bp,314.22bp)  .. node[fill=white, sloped] {$\frac{1}{13}$} (5);
  \draw [] (3) ..controls (174.76bp,101.33bp) and (188.03bp,150.39bp)  .. node[fill=white, sloped] {$\frac{1}{13}$} (4);
  \draw [] (3) ..controls (99.713bp,78.479bp) and (72.773bp,87.749bp)  .. node[fill=white, sloped] {$\frac{1}{13}$} (6);
  \draw [] (4) ..controls (212.42bp,239.85bp) and (226.02bp,289.34bp)  .. node[fill=white, sloped] {$\frac{1}{13}$} (5);
  \draw [] (4) ..controls (141.29bp,213.76bp) and (122.55bp,219.77bp)  .. node[fill=white, sloped] {$\frac{1}{13}$} (7);
  \draw [] (5) ..controls (178.61bp,350.91bp) and (162.83bp,355.3bp)  .. node[fill=white, sloped] {$\frac{1}{13}$} (8);
  \draw [] (6) ..controls (40.563bp,146.31bp) and (52.624bp,194.22bp)  .. node[fill=white, sloped] {$\frac{1}{13}$} (7);
  \draw [] (7) ..controls (76.553bp,282.04bp) and (90.14bp,328.31bp)  .. node[fill=white, sloped] {$\frac{2}{13}$} (8);
\end{tikzpicture}
    \caption{By allowing variable capacities, we are able to achieve $\sum_i\eta_i=e$.}
    \label{fig:3x3_optimal_dual}
\end{figure}
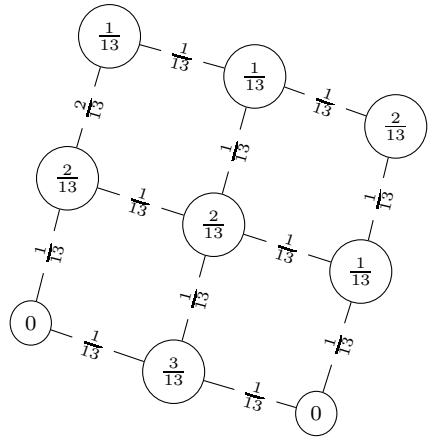

The conclusion is that in the absence of symmetry, one cannot expect optimal solutions with equal capacities.

\subsection{Weighted isoperimetric inequality}

Let $G = (V, E)$ be a graph.
For each $T \subset V$ we let
\[
  \partial T = E[T, V \setminus T]
\]
be the \df{edge boundary} of $T$.
As before, $E[T, V\setminus T] = (T \times (V\setminus T) \cup (V\setminus T)\times T) \cap E$.
The main insight in the solution of the dual problem for the cube graph $Q_3$ was the following inequality.
\[
  \sum_{i \in T} \sigma_{k, i} \leq \sum_{ij \in \partial T} \kappa_{ij}. 
\]
Clearly, if $ij \in E[T, T]$, then $\varphi_{k, ij}$ appears with a positive sign in $\sigma_{k, j}$ and with a negative sign in $\sigma_{k, i}$, hence the flow on the edges inside $T$ cancels out.
\[
  \sum_{i \in T} \sigma_{k, i} =
  \sum_{i \in T} \left( \sum_{mi \in E, m \in V} \varphi_{k, mi} - \sum_{im \in E, m \in V} \varphi_{k, im} \right)
  = \]
\[
  = \sum_{ij \in E[V \setminus T, T]} \varphi_{k, ij} - \sum_{ij \in E[T, V \setminus T]} \varphi_{k, ij}
  \leq \sum_{ij \in \partial T} \kappa_{ij}.
\]
In other words, the inequality states that the total supply in T is bounded by the total capacity at the boundary of $T$. 

Let $\mathcal{S} = \{ S_i \}_{i \in V}$ be a scale on $G$.
For each $k \in V$ and each $T \subset \overline{S}_k$ the flow $\varphi_k$
satisfies the inequality
\[
  \sum_{i \in T} \eta_i \leq \sum_{i \in T} \sigma_{k, i}.
\]
Hence for each flow $\varphi_k$, $k \in V$ and each $T \subset \overline{S}_k$ we have
\[
  \sum_{i \in T} \eta_i \leq \sum_{ij \in \partial T} \kappa_{ij}.
\]
This is a \df{weighted isoperimetric inequality}: the volume of $T$ with respect to $\eta$ is bounded by the volume of $\partial T$ with respect to $\kappa$. We have just proven the following theorem.

\begin{theorem}
  Each admissible solution of the dual problem satisfies a weighted isoperimetric inequality for all subsets of elements of the dual scale.

  More precisely, if ($\eta_i$, $\kappa_{ij}$, $\varphi_{k, ij}$) is an admissible solution of Linear Problem~\ref{lp:pseudo-flows}, then
  \[
    \sum_{i \in T} \eta_i \leq \sum_{ij \in \partial T} \kappa_{ij} 
  \]
  for each $T \subset \overline{S} \in \mathcal{\overline{S}}$.
\end{theorem}

Therefore the following problem is a relaxation of the pseudo-flows problem.

\begin{lp}{Weighted isoperimetric inequalities at scale $\mathcal{\overline{S}}$}\label{lp:isoperimetric}
\begin{maxi*}[1]
	{}{\sum_{i \in V} \eta_i \hspace{8cm}}{}{}
	\addConstraint{\sum_{ij \in E} \kappa_{ij}}{\leq 1}{}
	\addConstraint{\sum_{i \in T} \eta_i}{\leq \sum_{ij \in \partial T} \kappa_{ij}}{\text{ for each } T \subset \overline{S} \in \mathcal{\overline{S}}}
  \addConstraint{\kappa_{ij} \geq 0, \eta_i \in \mathbb{R}}{}{}
\end{maxi*}
\end{lp}

Our goal in this section is to show that both problems are equivalent.

\subsection{Partitions of unity supported by \texorpdfstring{$\mathcal{S}$}{S}}
Consider the following optimization problem.

\begin{op}{Partition of unity supported by $\mathcal{S}$}\label{op:partition}
  Minimize $\varepsilon \geq 0$ such that there exists
  a partition of unity $\{ \phi_i \}_{i \in V}$ for $G$ at scale $\mathcal{S}$ with $\varepsilon$-variation, i.e.
  \begin{enumerate}
      \item the $\phi_i$'s are non-negative, i.e.
      \[
        \phi_i(j) \geq 0 \text{ for each } i, j \in V,
      \]
      \item the sum of the $\phi_i$'s is a constant unit function, i.e.
      \[
        \sum_{i \in V} \phi_i(j) = 1 \text{ for each } j \in V,
      \]
      \item the sum of the variations of the $\phi_i$'s on the edge $kj \in E$ does not exceed $\varepsilon$, i.e.
      \[
        \sum_{i \in V} | \phi_i(k) - \phi_i(j) | \leq \varepsilon \text{ for each } kj \in E,\text{ and}
      \]
      \item each $\phi_i$ is supported by $\mathcal{S}$, i.e.
      \[
      \supp \phi_i \subset S_i \text{ for each } i \in V.
      \]
  \end{enumerate}
\end{op}

Optimization Problem~\ref{op:partition} has the following linear programming formulation.

\begin{lp}{Partition of unity supported by $\mathcal{S}$}\label{lp:partition}
\begin{mini*}[1]{}{e \hspace{8cm}}{}{}
\addConstraint{\sum_{j \in V} f_{j, i}}{= 1}{\text{ for each } i \in V}
\addConstraint{f_{k, i} - f_{k, j}}{ \leq e_{i, kj}}{\text{ for each } ij \in E, k \in V}
\addConstraint{f_{k, j} - f_{k, i}}{ \leq e_{i, kj}}{\text{ for each } ij \in E, k \in V}
\addConstraint{\sum_{i \in V} e_{i, kj}}{\leq e}{\text{ for each } kj \in E}
\addConstraint{f_{i, j}}{= 0}{\text{ for each } i \in V, j \in V \setminus S_i}
\addConstraint{f_{i, j}, e_{i, kl} \geq 0, e}{\in \mathbb{R}}{\text{ for each } i,j \in V, kl \in E}
\end{mini*}
\end{lp}

\begin{proposition}\label{prop:partitions}
  Linear Problem~\ref{lp:measures} at scale $\mathcal{S}$ is equivalent to Linear Problem~\ref{lp:partition} at scale $\mathcal{\overline{S}}$, i.e. 
  we can pair admissible solutions of both problems
  by setting 
  \[
    x_{i, j} = f_{j, i}.
  \]
\end{proposition}

To see the equivalence on an example, consider the table in Section~\ref{subsec:dual} with values of~$\xi$.
In the $i$th column, values of the probability measure $\xi_i$ supported by $S_i$ are listed.
The rows of the table form a partition of unity and are supported by the dual scale.

\subsection{Partitions of unity subordinated to \texorpdfstring{$\mathcal{S}$}{S}}

In the last section, the elements $\phi_i$ of a partition of unity were in one-to-one correspondence with the elements $S_i$ of the cover $\mathcal{S}$.
Usually this is not required.
For our application we consider an equivalent problem without the one-to-one correspondence and with a non-standard requirement that functions that form the partition of unity are flat, i.e. constant on their support.

\begin{op}{Flat partition of unity}\label{op:flat partition}
  Minimize $\epsilon \geq 0$ such that there exists
  a finite  partition of unity $\{ \psi_\alpha \}$ for $G$ subordinated to scale $\mathcal{S}$ with $\epsilon$-variation and flat elements, i.e.
  \begin{enumerate}
      \item each $\psi_\alpha$ is non-negative, i.e. for each $\alpha$
      \[
        \psi_\alpha(i) \geq 0 \text{ for each } i \in V,
      \]
      \item the elements of the partition are constant on their support, i.e.
      \[ 
      \psi_\alpha(j) = \psi_\alpha(k)
      \]
      for each $\alpha$ and each $j, k \in V$ such that $\psi(j), \psi(k) > 0$,
      \item the sum of the $\psi$'s is a constant unit function, i.e.
      \[ 
      \sum_\alpha \psi_\alpha(i) = 1 \text{ for each } i \in V,
      \]
      \item the sum of the variations of the $\psi$'s on each edge $kj \in E$ does not exceed $\epsilon$, i.e.
      \[ \sum_\alpha | \psi_\alpha(k) - \psi_\alpha(j) | \leq \varepsilon \text{ for each } kj \in E,\text{ and}
      \]
      \item each $\psi$ is supported by $\mathcal{S}$, i.e. 
      \[
      \supp \psi_\alpha \subset S \text{ for each } \alpha \text{ and some } S \in \mathcal{S}.
      \]
  \end{enumerate}
\end{op}

\begin{lemma}\label{lem:partitions1} We have
\[
\varepsilon \leq \epsilon.
\]
That is, for a graph $G$ at scale $\mathcal{S}$ the optimal solution of Optimization Problem~\ref{op:partition} is bounded by the optimal solution of Optimization Problem~\ref{op:flat partition}.
\end{lemma}
\begin{proof}
Let $\{\psi_\alpha\}$ be a finite partition of unity subordinated to $\mathcal{S} = \{ S_i \}_{i \in V}$ where the functions $\psi_\alpha$ only assume one nonzero value. For every $\alpha\in A$, we choose a vertex $i=i(\alpha)$ with the property that $\supp(\psi_\alpha)\subseteq S_i$. 

We put 
\[
  \phi_i(j)=\sum_{\alpha \colon i=i(\alpha)}\psi_{\alpha}(j). 
\]

Then, if $j\in \supp(\phi_i)$, then there is some $\alpha$ such that $i(\alpha)=i$ and so $j\in \supp(\psi_\alpha)$, which is to say $j\in S_i$. Next, 
\[
\sum_{i\in V}\phi_i(j)=\sum_{i\in V}\left(\sum_{\alpha\colon i=i(\alpha)}\psi_\alpha(j)\right)=\sum_{\alpha}\psi_\alpha(j)=1.
\]

Finally, if we fix an edge, $jk$, then 
\begin{multline*}
\sum_{i\in V}|\phi_i(j)-\phi_i(k)|=\sum_{i\in V}\left|\sum_{\alpha\colon i=i(\alpha)}\psi_\alpha(j)-\sum_{\alpha\colon i=i(\alpha)}\psi_\alpha(k)\right|\\ 
\le \sum_{i\in V}\sum_{\alpha\colon i=i(\alpha)} |\psi_\alpha(j)-\psi_\alpha(k)|\le \varepsilon.
\end{multline*}
\end{proof}

Note that in the proof we do not use the assumption that the $\psi_\alpha$'s are flat.
We will need this variant of partitions of unity in the next section.
We also have the reversed inequality.

\begin{lemma}\label{lem:partitions2}
We have
\[
\varepsilon \geq \epsilon.
\]
That is, for graph $G$ at scale $\mathcal{S}$ the optimal solution of Optimization Problem~\ref{op:flat partition} is bounded by the optimal solution of Optimization Problem~\ref{op:partition}.
\end{lemma}
\begin{proof} Let $\{\phi_i\}$ be a partition of unity with $\supp(\phi_i)\subset S_i$ and $\sum_{i\in V}|\phi_i(j)-\phi_i(k)|\le \varepsilon$ whenever $jk$ is an edge. 

For each $i\in V$, the function $\phi_i$ takes on only finitely many (say $n_i$-many) values: $0=y^i_0<y^i_1<\cdots<y^i_{n_i}$. Put 

\[\psi_{\alpha^i_r}(j) = \begin{cases} 0 &\mbox{if } \phi_i(j)\le y^i_{r-1}; \\
y^i_r-y^i_{r-1} & \mbox{otherwise.}\end{cases}\] 

We observe that for each $i$ and each $r\in \{1,\ldots, n_i\}$, the function $\phi_{\alpha^i_r}$ takes on exactly one nonzero value. Moreover, it is easy to see that $\supp(\psi_{\alpha^i_r})\subset S_i$ and that for every $j$, if we sum over all values of $\alpha$, we obtain: \[\sum_{\alpha}\psi_{\alpha}(j)=\sum_{i\in V}\sum_{r=1}^{n_i}\psi_{\alpha^i_r}(j)=\sum_{i\in V}\phi_i(j)=1.\]

Finally, it is easy to verify that, \[|\phi_i(j)-\phi_i(k)|=\sum_{r=1}^{n_i}|\psi_{\alpha^i_r}(j)-\psi_{\alpha^i_r}(k)|,\] and so if $jk$ is an edge, then
\[\sum_{\alpha}|\psi_{\alpha}(j)-\psi_{\alpha}(k)|=\sum_{i\in V}\sum_{r=1}^{n_i}|\psi_{\alpha^i_r}(j)-\psi_{\alpha^i_r}(k)|=\sum_{i\in V}|\phi_i(j)-\phi_i(k)|\le \varepsilon.\]

\end{proof}

As a corollary we obtain the following theorem.
\begin{theorem}\label{thm:measures equals partition}
Let $G=(V,E)$ be a finite graph, let $\varepsilon>0$, and let $\mathcal{S}=\{S_i\}_{i\in V}$ be a scale. As before, for $S_i\in \mathcal{S}$, we let $\bar{S}_i=\{j\in V\colon j\in S_i\}$. The following are equivalent:
\begin{enumerate}
    \item there exists a collection $\{\xi_i\}_{i\in V}$ of probability measures on $G$ with $\varepsilon$ variation at scale $\mathcal{S}$;
    \item there exists a partition of unity $\{\phi_{i}\}_{i\in V}$ on $G$ with $\varepsilon$ variation at scale $\overline{\mathcal{S}}$; and 
    \item there exists a flat partition of unity $\{\psi\}$ on $G$ with $\varepsilon$ variation subordinated to $\overline{\mathcal{S}}$.
\end{enumerate}
\end{theorem}
\begin{proof}
  Equivalence $(1) \Leftrightarrow (2)$ follows from Proposition~\ref{prop:partitions}.
  Equivalence $(2) \Leftrightarrow (3)$ follows from Lemma~\ref{lem:partitions1} and Lemma~\ref{lem:partitions2} applied to the dual scale $\mathcal{\overline{S}}$.
\end{proof}
\subsection{The dual problem to maximal weighted isoperimetric number}
We are ready to construct a dual problem to Linear Problem~\ref{lp:isoperimetric}
  and to show how its optimal solution yields a feasible solution for Optimization Problem~\ref{op:flat partition}.

\begin{lp}{Weighted isoperimetric inequalities dual at scale $\mathcal{\overline{S}}$}\label{lp:isoperimetric dual}
\begin{mini*}[1]
  {}{a}{}{}
  \addConstraint{\sum_{T \colon ij \in \partial T} z_T}{\leq a}{\text{ for each } ij \in E}
  \addConstraint{\sum_{T \colon i \in T} z_T}{=1}{\text{ for each } i \in V}
\end{mini*}
\end{lp}

\begin{lemma}\label{lem:isoperimetric dual}
  Linear Problem~\ref{lp:isoperimetric dual} is dual to Linear Problem~\ref{lp:isoperimetric}.
\end{lemma}
\begin{proof}
This follows from the dual problem derivation:
\[
\sum_{i \in V} \eta_i \leq
\sum_{i \in V} \eta_i + 
\sum_{T \colon T \subset \overline{S} \in \mathcal{\overline{S}}} z_T 
(\sum_{ij \in \partial T} \kappa_{ij} - \sum_{i \in T} \eta_i) + a(1 - \sum_{ij \in E} \kappa_{ij}) =
\]
\[
= a + \sum_{ij \in E} \kappa_{ij}(-a + \sum_{T \colon ij \in \partial T} z_T)
  + \sum_{i \in V} \eta_i (1 - \sum_{T \colon i \in T} z_T) \leq a.
\]
\end{proof}

\begin{lemma}\label{lem:isoperimetric dual to partition}
  Let $({\bf z}, a)$ be an admissible solution of Linear Problem~\ref{lp:isoperimetric dual}.
  For each $T$ let
  \[
    \psi_T(i) = \left\{
    \begin{array}{ll}
      z_T & \text{ if } i \in T, T \subset \overline{S} \in \mathcal{\overline{S}} \\
      0 & \text{ otherwise }
    \end{array}
    \right.
  \]
  Then $\{ \psi_T \}$ is a flat partition of unity subordinated to $\mathcal{\bar S}$ with $a$-variance.
\end{lemma}
\begin{proof}
  Conditions (1), (2) and (5) of Optimization Problem~\ref{op:flat partition} follow directly from the definition of $\psi_T$.
  
  Condition (3):
  \[
    \sum_T \psi_T(i) = \sum_{T \colon i \in T} z_T = 1.
  \]
  
  Condition (4):
  \[
  \sum_T | \psi_T(i) - \psi_T(j) | = \sum_{T \colon ij \in \partial T} z_T \leq a.
  \]
\end{proof}

\begin{theorem}\label{thm:partition of unity dual}
  The optimal solution of Linear Problem~\ref{lp:isoperimetric} is equal to the optimal solution of Linear Problem~\ref{lp:pseudo-flows}.
\end{theorem}
\begin{proof}
  Linear Problem~\ref{lp:isoperimetric} is a relaxation of Linear Problem~\ref{lp:pseudo-flows}, so we have the ``$\geq$'' inequality between optimal values.
  
  Linear Problem~\ref{lp:pseudo-flows} is dual to Linear Problem~\ref{lp:measures}, which is equivalent to Optimization Problem~\ref{op:flat partition}. By Lemma~\ref{lem:isoperimetric dual to partition}, the value of $a$ in Linear Problem~\ref{lp:isoperimetric dual} bounds the optimal value of Optimization Problem~\ref{op:flat partition} from above. By Lemma~\ref{lem:isoperimetric dual}, Linear Problem~\ref{lp:isoperimetric dual} is dual to Linear Problem~\ref{lp:isoperimetric}, which establishes the ``$\leq$'' inequality.
\end{proof}

\subsection{Heawood graph}\label{sec:heawood}

Linear Problem~\ref{lp:pseudo-flows} has more variables than Linear Problem~\ref{lp:isoperimetric} but a lot fewer constraints. Usually it is easier to construct flows than to check exponentially many weighted isoperimetric inequalities. There are exceptions to this rule though.

The Heawood graph is the following $3$-regular graph on $14$ vertices with the property that the length of the shortest cycle (i.e. girth) is $6$.
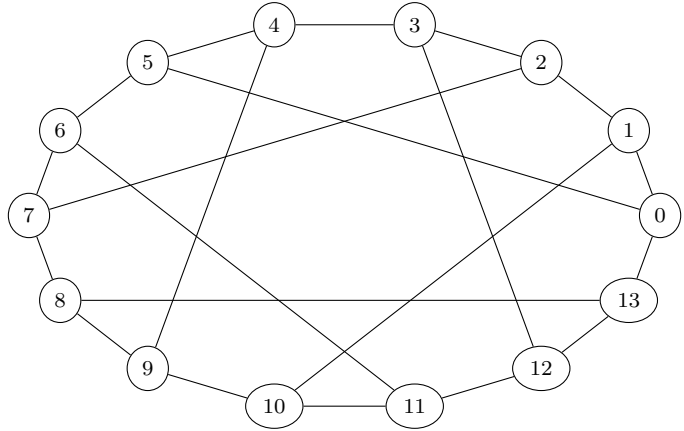
\begin{center}
  \includegraphics[width=0.7\textwidth]{figures/iso/graph.tikz}
\end{center}
We will solve Linear Problem~\ref{lp:isoperimetric} for this graph at scale $S = 2$. Because of the symmetry of the graph it is reasonable to look for a solution with equal edge capacities and equal node demands (see Theorem~\ref{thm:averaging}).
The graph has $21$ edges and so we set $\kappa_{ij} = \frac 1{21}$ for each $ij \in E$.
Because the girth of the graph is $6$ and the scale is $2$, for each $T \subset S \subset \mathcal{S}$ the subgraph spanned by $T$ is a disjoint union of trees.
It is enough to check weighted isoperimetric inequalities for connected subgraphs.
A subtree with $n$ vertices in a $3$-regular graph has $n-1$ inside edges and $3n - 2(n-1) = n + 2$ edges on the boundary.
Weighted isoperimetric inequalities are of the form
\[
  n \cdot \eta \leq (n+2) \cdot \frac 1{21}.
\]
A $2$-ball in the Heawood graph has $10$ vertices so $n \leq 10$. Since
$\eta \leq \frac 1{21} + \frac 2{21 n}$, we have $\eta \leq \frac 1{21} + \frac 2{210} = \frac 2{35}$.
We get the following solution of Linear Problem~\ref{lp:isoperimetric}.

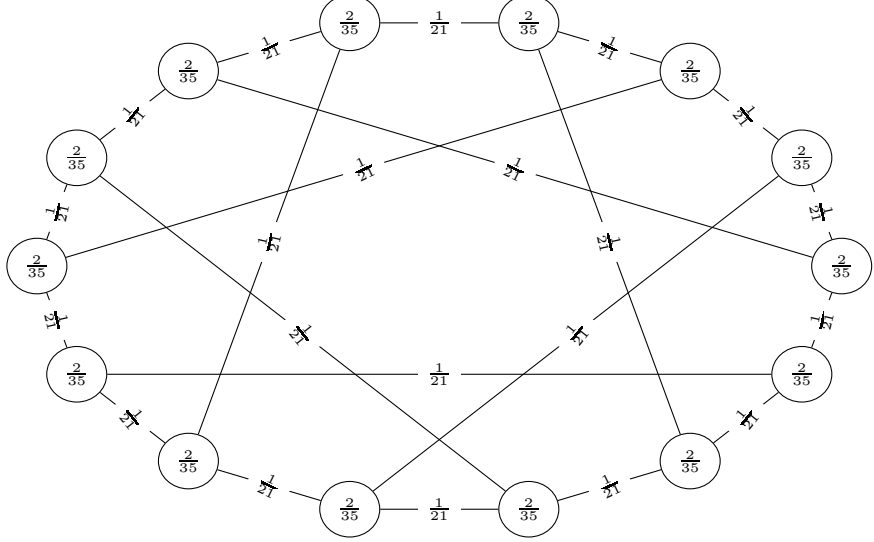
\begin{center}
  \includegraphics[width=0.9\textwidth]{figures/iso/capacity.tikz}
\end{center}

The optimal objective is $14 \cdot \frac 2{35} = \frac 45$.
The optimality of the solution may be verified by the observation that the solution  of the primal problem with probability measures that are constant on their supports has $\varepsilon = \frac 45$.

\section{Lift and project}\label{sec:lift}

We let $F$ denote the set of admissible solutions of Linear Problem~\ref{lp:pseudo-flows}
for a graph $G = (V, E)$ at dual scale $\mathcal{\overline{S}}$, i.e. we let
\[
F = \left\{ 
(\eta, \kappa, \varphi) \in \mathbb{R}^n \times \mathbb{R}^m \times \mathbb{R}^l \colon
\begin{array}{rrl}
& \sum_{ij \in E} \kappa_{ij} & \leq 1 \\
\forall_{ij \in E} & -\kappa_{ij} \leq \varphi_{k, ij} & \leq \kappa_{ij}\\
\forall_{k \in V, i \in \overline{S}_k} & \sigma_{k, i} & \geq \eta_i
\end{array}
\right\}
\]
where
\[
  \sigma_{k, i} = \sum_{m \in V\colon mi \in E} \varphi_{k, mi} - \sum_{m \in V\colon im \in E} \varphi_{k, im}
\]
and $n = |V|$, $m = |E|$, $l = |V|\cdot |E|$. For $k \in V$ we let the flow $\varphi_k \colon E \to \mathbb{R}$ be defined on all of $E$ even though only the values on $E[\overline{S}_k, V]$ are relevant.

We let $K$ denote the set of admissible solutions of Linear Problem~\ref{lp:isoperimetric}, i.e. we let
\[
K = \left\{ 
( \eta, \kappa ) \in \mathbb{R}^n \times \mathbb{R}^m \colon 
\begin{array}{rrl}
  & \sum_{ij \in E} \kappa_{ij} & \leq 1\\
  \forall_{T \subset \overline{S} \in \mathcal{\overline{S}}} & \sum_{i \in T} \eta_i & \leq \sum_{ij \in \partial T} \kappa_{ij} \\
  \forall_{ij \in E} & \kappa_{ij} & \geq 0
\end{array}
\right\}.
\]

In the previous section we proved that the maximal value of $\sum_{i \in V} \eta_i$ is equal on $K$ and $F$. In the present section we show that in fact $F$ is a lift of $K$.

\begin{theorem}\label{thm:lift and project}
  The projection
  \[
    \pi \colon \mathbb{R}^n \times \mathbb{R}^m \times \mathbb{R}^l \to \mathbb{R}^n \times \mathbb{R}^m
  \]
  maps $F$ onto $K$, i.e. for each $\eta$ and $\kappa$ that satisfy weighted isoperimetric inequalities there exists a family of flows $\varphi_k$ that satisfies demand $\eta$ under capacity $\kappa$.
\end{theorem}

\begin{proof}
  Let $(\eta, \kappa) \in K$.
  We will show that for each $k \in V$ there exists a flow $\varphi_k \colon E \to \mathbb{R}$ such that
  \[
    \forall_{ij \in E}\ -\kappa_{ij} \leq \varphi_{k, ij} \leq \kappa_{ij}
  \]
  and
  \[
    \forall_{i \in \overline{S}_k}\ \sigma_{k, i} \geq \eta_i.
  \]
  
  Consider
  \[
  \varPhi = \left\{
  \varphi \colon E \to \mathbb{R} \colon \forall_{ij \in E}\ -\kappa_{ij} \leq \varphi_{ij} \leq \kappa_{ij}
  \right\}.
  \]
  Clearly, $\varPhi$ is a compact subset of $\mathbb{R}^m$.

  Fix $k \in V$. There exists $\varphi \in \varPhi$ such that the value of
    \begin{equation}\label{eqn:eta-sigma}\tag{*}
          \sum_{i \in \overline{S}_k} \max\left\{0, \eta_i - \left(\sum_{m \in V\colon mi \in E} \varphi_{mi} - \sum_{m \in V\colon im \in E} \varphi_{im}\right) \right\}
  \end{equation}
  is minimal. 
  If the value of the sum \eqref{eqn:eta-sigma} is $0$, then we put $\varphi_k = \varphi$ and we are done.
  Assume therefore that the value of the sum \eqref{eqn:eta-sigma} is positive.
  
  For each $i \in V$ we let
  \[
    \sigma_i = \sum_{m \in V\colon mi \in E} \varphi_{mi} - \sum_{m \in V\colon im \in E} \varphi_{im}
  \]
  be the supply of the pseudo-flow $\varphi$ at node $i$.

  Let $T \subset V$ and let $i_0 \in T$. Let $j \in V$, $j \neq i_0$ ($j$ may be outside of $T$). We say that the flow from $j$ to $i_0$ is \df{not saturated} (in $T$) if there exists a sequence $i_0, i_1, i_2, \ldots, i_k = j$, such that for each $0 \leq m < k$ we have $i_m \in T$ and either
  \[
    i_{m}i_{m+1} \in E \text{ and } \varphi_{i_{m}i_{m+1}} > -\kappa_{i_{m}i_{m+1}}
  \]
  or
  \[
    i_{m+1}i_{m} \in E \text{ and } \varphi_{i_{m+1}i_{m}} < \kappa_{i_{m+1}i_{m}}.
  \]
  
  Observe that if a flow from $j$ to $i_0$ is not saturated, then there exist a $\delta > 0$
  and a $\varphi' \in \varPhi$ with supply $\sigma_{i_0} + \delta$ at $i_0$, with supply $\sigma_j - \delta$ at $j$, and with all other supplies unchanged (we flow an additional $\delta$ along the path from $j$ to $i_0$ and pick $\delta$ small enough such that the edge capacities are not exceeded).
  
  Since the value of \eqref{eqn:eta-sigma} is positive, there exists $i_0 \in \overbar{S}_k$
    such that $\sigma_{i_0} < \eta_{i_0}$.
  Let $T \subset \overline{S}_k$ be a set with maximal cardinality that satisfies the following conditions:
  \begin{enumerate}
    \item $i_0 \in T$ and
    \item for each $j \in T$, $j \neq i_0$ , the flow from $j$ to $i_0$ is not saturated.
  \end{enumerate}
  Let $j \in T, j \neq i_0$. If $\sigma_j > \eta_j$, then the flow $\varphi'$ with supply $\sigma_{i_0} + \delta$ at $i_0$ and supply $\sigma_j - \delta$ at $j$ (with other supplies unchanged) would have smaller value of \eqref{eqn:eta-sigma} for small enough $\delta$. Hence $\sigma_j \leq \eta_j$ for each $j \in T$. Therefore the total supply in $T$ does not meet the total demand, i.e. 
  \[
  \sum_{i \in T} \sigma_i < \sum_{i \in T} \eta_i.
  \]

  Since $T \subset \overbar S_k \in \mathcal{\overbar S}$, we have a weighted isoperimetric inequality for $T$, which combined with (1) and with the definition of $\sigma_i$ gives
  \[
     \sum_{ij \in E[V \setminus T, T]} \varphi_{ij} - \sum_{ij \in E[T, V \setminus T]} \varphi_{ij}  = \sum_{i \in T} \sigma_i < \sum_{i \in T} \eta_i \leq \sum_{ij \in \partial T} \kappa_{ij}.
  \]
  So there exists an edge $ij \in \partial T$ such that the flow along $ij$ {\bf into} $T$ is not maximal.

  Without any loss of generality we may assume that $i \in T$ and $j \not\in T$.
  From the choice of $ij$, the flow from $j$ to $i_0$ is not saturated in $T$.
  If $j \not\in \overbar{S}_k$, then we have no demand constraint for $j$ and we may flow some $\delta > 0$ from $j$ to $i_0$, contradicting the assumption that $\varphi$ minimizes \eqref{eqn:eta-sigma}.
  If $j \in \overbar{S}_k$, then the set $T \cup \{ j \}$ satisfies conditions (1) and (2), contradicting the maximality of $T$. The contradiction shows that the value of \eqref{eqn:eta-sigma} had to be $0$.
\end{proof}

\section{Averaged solutions}

The set of optimal solutions of a linear problem is a convex polyhedron.
Hence the center of mass of a set of optimal solutions is an optimal solution.
In Linear Problem~\ref{lp:pseudo-flows}, if we relabel the vertices of $G$ using an automorphism, we get a new solution.
If there is enough symmetry in the graph, this guarantees that we can find an optimal solution with equal edge capacities and/or equal node supplies. 

In the following theorem, we do not require the automorphisms to preserve the orientation of $G$. Indeed, when $ij\in E$ then for the purposes of what follows, we may define $\kappa_{ji}$ to be equal to $\kappa_{ij}$.

\begin{theorem}\label{thm:averaging}
	Let $G=(V,E)$ be a graph. Let $\Gamma$ be a group that acts on $G$ by automorphisms (here we allow orientation of edges to be reversed by an automorphism). 
	Let $\{ S_i \subset V \colon i \in V\}$ be a family of subsets of $V$ that is invariant under the action of $\Gamma$, i.e.
	\[
	\gamma(S_i) = S_{\gamma(i)}
	\]
	for each $\gamma \in \Gamma$ and $i \in V$.
	
	If $\Gamma$ acts transitively on the vertices of $G$, then there exists an optimal solution of Linear Problem~\ref{lp:isoperimetric} at scale $\mathcal{S}$ such that $\eta_i = \eta_j$ for each $i, j \in V$.
	
	If $\Gamma$ acts transitively on the edges of $G$, then there exists an optimal solution of Linear Problem~\ref{lp:isoperimetric} at scale $\mathcal{S}$ such that for each $ij \in E$, $\kappa_{ij} = \frac 1{|E|}$.
	
	If $\Gamma$ acts transitively on both the edges and the vertices of $G$, then there exists an optimal solution such that $\eta_i = \eta_j$ for each $i, j \in V$ and $\kappa_{ij} = \frac 1{|E|}$ for each $ij \in E$.
\end{theorem}

By Theorem~\ref{thm:lift and project}, we can replace Linear Problem~\ref{lp:isoperimetric} by Linear Problem~\ref{lp:pseudo-flows} in the statement of Theorem~\ref{thm:averaging}.

\begin{proof}
 
  Let $K$ be a set of admissible solutions of Linear Problem~\ref{lp:isoperimetric} (see Section~\ref{sec:lift}).
  Let $(\eta, \kappa) \in K$ be an optimal solution.
  Without any loss of generality we may assume that $\sum_{ij \in E} \kappa_{ij} = 1$.
  For $\gamma \in \Gamma$, let
  \[
    (\gamma.\eta, \gamma.\kappa) = (\eta \circ \gamma^{-1}, \kappa \circ (\gamma^{-1} \times \gamma^{-1}))
  \]
  be a solution $(\eta, \gamma)$ translated by the action of $\gamma$ on $G$. Since $\gamma \times \gamma$ is a bijection on $E$, $\gamma.\kappa$ is well defined and
  \[
    \sum_{ij \in E} (\gamma.\kappa)_{ij} = \sum_{ij \in E} \kappa_{\gamma^{-1}(i)\gamma^{-1}(j)} = \sum_{ij \in E} \kappa_{ij} = 1.
  \]
  As we observed above, we do not require the orientation of the edges 
  to be preserved. In the case that $\gamma$ inverts the edge $ij$, we give 
  $\kappa_{ji}$ same value as 
  $\kappa_{ij}$. 
  Since $\gamma$ is a bijection on $V$, we have
  \[
    \sum_{i \in V} (\gamma.\eta)_i = \sum_{i \in V} \eta_{\gamma^{-1}(i)} = \sum_{i \in V} \eta_i.
  \]
  Let $T \subset S_j \in \mathcal{S}$. Let $k = \gamma^{-1}(j)$. We have
  \[
    \sum_{i \in T} (\gamma.\eta)_i = \sum_{i \in T} \eta_{\gamma^{-1}(i)}
    = \sum_{i \in \gamma^{-1}(T)} \eta_i \leq
  \]
  \[
    \leq \sum_{im \in \partial \gamma^{-1}(T)} \kappa_{im} =
    \sum_{im \in \gamma^{-1}(\partial T)} \kappa_{im} =
    \sum_{im \in \partial T} \kappa_{\gamma^{-1}(im)} =
    \sum_{im \in \partial T} (\gamma.\kappa)_{im}.
  \]
  The inequality follows from the weighted isoperimetric inequality of the problem and the fact that $\gamma^{-1}(T) \subset \gamma^{-1}(S_j) = S_k$.
  Therefore $(\gamma.\eta, \gamma.\kappa) \in K$ and if $(\eta, \kappa)$ is an optimal solution, then so is $(\gamma.\eta, \gamma.\kappa)$.
  Let
  \[
    (\overline{\eta}, \overline{\kappa}) = \frac 1{|\Gamma|} \sum_{\gamma \in \Gamma} (\gamma.\eta, \gamma.\kappa).
  \]
  By the convexity of the set of optimal solutions, $(\overline{\eta}, \overline{\kappa})$ is an optimal solution.
  
  Let $\Gamma.i$ denote the orbit of $i \in V$ under the action of $\Gamma$. If $j, k \in \Gamma.i$, then
  \[
    \overline{\eta}_j = \frac 1{|\Gamma.i|} \sum_{l \in \Gamma.i}  \eta_l = \overline{\eta}_k,
  \]
  because stabilizers of all vertices in $\Gamma.i$ under the action of $\Gamma$ are of equal size.
  Hence $\overline{\eta}$ is constant on orbits of the action of $\Gamma$ and if the action is transitive, then all values of $\overline{\eta}$ are equal.
  
  Similarly, $\overline{\kappa}$ is constant on orbits of the action of $\Gamma$ and if the action is transitive, then all values of $\overline{\kappa}$ are equal. Since $\sum_{ij \in E} \overline{\kappa}_{ij} = 1$, in such case we have $\overline{\kappa}_{ij} = \frac 1{|E|}$.
\end{proof}

\subsection{The circular ladder graph \texorpdfstring{$\mathbb{Z}_2 \times \mathbb{Z}_7$}{}.}

The symmetry of a graph can be exploited even if the action of the isometry group is not transitive. 
We will show this for the Cayley graph of the group $\mathbb{Z}_2 \times \mathbb{Z}_7$ with the standard generating set $\{ (1, 0), (0, 1) \}$.

An optimal solution at scale $S = 1$ of Linear Problem~\ref{lp:isoperimetric} is shown below. The label on edge $ij$ is its capacity $\kappa_{ij}$; the label on vertex $i$ is its demand $\eta_i$.

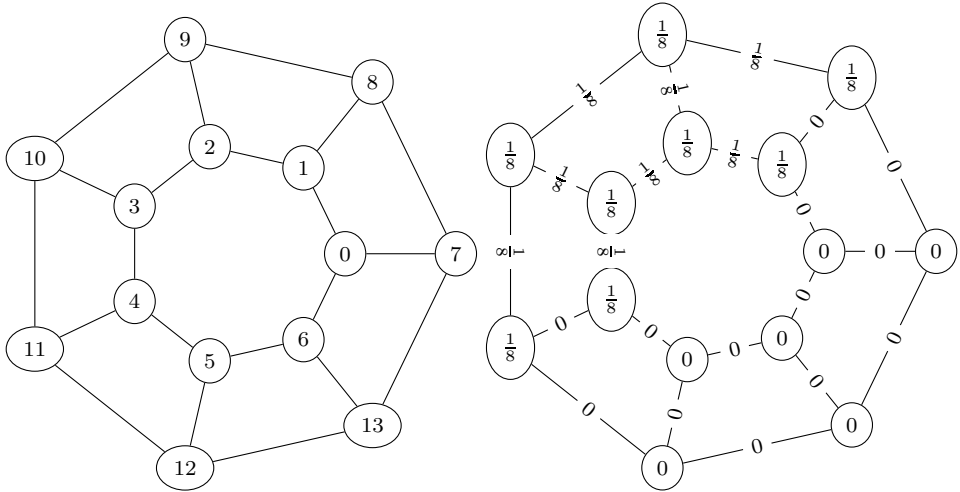
\begin{center}
  \includegraphics[width=0.49\textwidth]{figures/avg/graph.tikz}
  \includegraphics[width=0.49\textwidth]{figures/avg/capacity.tikz}
\end{center}

The objective value is $8 \cdot \frac 18 = 1$, which is a basic solution found by the simplex method, i.e. a vertex of the polyhedron of optimal solutions. 
Hence it lacks symmetry.
Using the symmetry of the graph we can find a simpler solution by hand.

The action of the isometry group is transitive on vertices, so we may put $\eta_i = \eta$ for each $i \in V$.
The action of the isometry group on edges has two orbits. 
Hence we may find a solution of Linear Problem~\ref{lp:isoperimetric} with capacity (say) $\kappa$ on edges corresponding to the $(1, 0)$ generator and capacity (say) $\lambda$ on edges corresponding to the $(0, 1)$ generator.

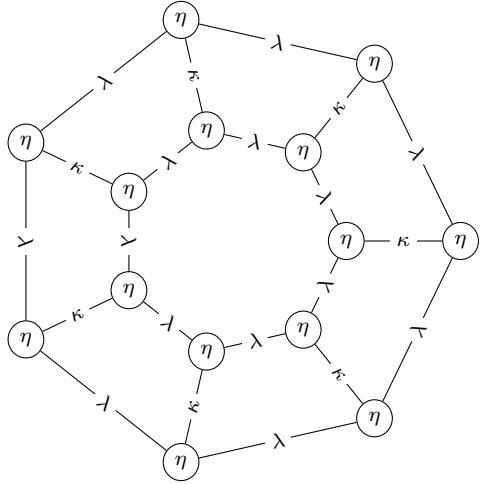
\begin{center}
  {\tiny
  \includegraphics[width=0.49\textwidth]{figures/avg/uniform_solution_symbolic.tikz}}
\end{center}

The capacity constraint is
\[
  14\lambda + 7\kappa \leq 1
\]
and without any loss of generality we may put $\kappa = \frac 17 - 2\lambda$.
The weighted isoperimetric inequalites are
\[
  {
  \renewcommand*{\arraystretch}{1.2}
  \begin{array}{rll}
  \eta & \leq 2\lambda + \kappa & = \frac 17 \\
  2\eta & \leq 2\lambda + 2\kappa & = \frac 27 - 2\lambda \\
  2\eta & \leq 4\lambda & = 4\lambda \\
  3\eta & \leq 4\lambda + \kappa & = \frac 17 + 2\lambda \\
  4\eta & \leq 4\lambda + 2\kappa & = \frac 27.
  \end{array}
  }
\]
The inequalities correspond to different connected subsets of a ball of radius $1$ in the graph and we omit duplicates.
For $\lambda = \frac 2{35}$ we have $\frac 17 - \lambda = \frac 1{21} + \frac 23\lambda \leq 2\lambda$. This gives an optimal bound $\eta = \frac 1{14}$ and  $\kappa = \frac 1{35}$. The solution is the following.
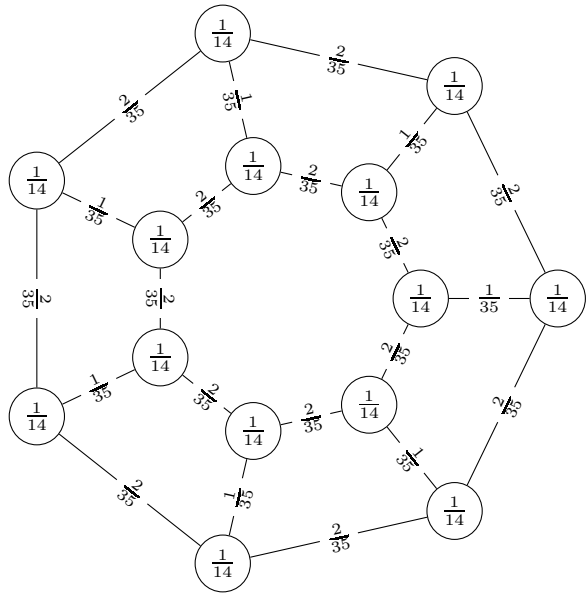
\begin{center}
  {\tiny
  \includegraphics[width=0.6\textwidth]{figures/avg/uniform_solution.tikz}}
\end{center}
Again, the objective value is $14 \cdot \frac 1{14} = 1$.

\section{Cheeger constant}

The capacity $\kappa$ and supply $\eta$ are implicit in the solution of Optimization Problem~\ref{op:pseudo-flows}. If pseudo-flows $\varphi_k$ are chosen, then the optimal values of $\kappa$ and $\eta$ can be easily determined.
However a natural way to solve the problem is first to fix $\kappa$
  and~$\eta$ and then look for an optimal pseudo-flow family, as we did in many of the examples above.
Choosing appropriate values for $\kappa$ and $\eta$ makes the problem hard to crack.
To make it easier we can place additional constraints on Linear Problem~\ref{lp:pseudo-flows}.

The main idea that is inspired by Theorem~\ref{thm:averaging} is to set all capacities to $\kappa_{ij} = \frac 1{|E|}$ and all supplies to $\eta_i = \eta$.
We add these extra constraints to
Linear Problem~\ref{lp:pseudo-flows} to get Linear Problem~\ref{lp:uniform}.

\begin{lp}{Uniform pseudo-flows at the dual scale $\mathcal{\overbar S}$}\label{lp:uniform}
	\begin{maxi*}[1]
		{}{|V|\cdot\eta \hspace{8cm}}{}{}
		\addConstraint{\varphi_{ij, k}}{\leq \frac 1{|E|}}{\text{ for each } ij \in E, k \in V}
		\addConstraint{-\varphi_{ij, k}}{\leq \frac 1{|E|}}{\text{ for each } ij \in E, k \in V}
		\addConstraint{
			\sum_{j \in V, ji \in E} \varphi_{ji, k} -
			\sum_{j \in V, ij \in E} \varphi_{ij, k}
		}{ \geq \eta
		}{\text{ for each } k \in V, i \in \overbar{S}_k}
		\addConstraint{\eta, \varphi_{ij, k}}{\in \mathbb{R}}{{} }
	\end{maxi*}
\end{lp}

The present section is devoted to the study of Linear Problem~\ref{lp:uniform}.
We will show how it is related to the Cheeger constant of the graph.
We also study the dual problem to Linear Problem~\ref{lp:uniform}, which is a relaxation of Linear Problem~\ref{lp:measures}.
As a result, we find that the expanding property, which is used to show that a Cayley graph of a finitely generated group doesn't have property~A, contradicts a property that is much weaker than property~A.

\subsection{Isoperimetric number and Cheeger constant}
\label{sec:cheeger}

The Cheeger constant is a graph invariant defined as
\[
  \gamma(G) = \min \left\{ \frac{|\partial T|}{|T|} \colon T \subset V, T \neq \emptyset, |T| \leq \frac{|V|}2 \right\}.
\]
It is of great importance in many areas of mathematics~\cite{avr2009}.
Below we define its variant \df{at scale $\mathcal{S}$}, which is relevant to the contents of this paper.

\begin{definition}\label{def:cheeger at scale S}
    Let $\mathcal{S}$ be a scale on a graph $G$.
	We let
	\[
	\gamma(G, \mathcal{S}) = \min \left\{ \frac{|\partial T|}{|T|} \colon T \subset S \in \mathcal{S}, T \neq \emptyset \right\}
	\]
	be the \df{Cheeger constant} of $G$ \df{at scale $\mathcal{S}$}.
    For $T \neq \emptyset$, $T\subset V$ we call
    \[
    \varphi(T) = \frac{|\partial T|}{|T|}
    \]
    the \df{isoperimetric number of $T$}.
\end{definition}

\begin{theorem}\label{thm:cheeger}
  The optimal solution of Linear Problem~\ref{lp:uniform} at scale $\mathcal{\overbar S}$ is equal to 
  \[
    \frac{|V|}{|E|} \gamma(G, \mathcal{\overbar S}),
  \]
  i.e. the Cheeger constant of $G$ at scale~$\mathcal{\overbar S}$ multiplied by $\frac{|V|}{|E|}$.
\end{theorem}
\begin{proof}
  By Theorem~\ref{thm:lift and project}, we may project Linear Problem~\ref{lp:uniform} and replace the flows with weighted isoperimetric inequalities.
  \begin{maxi*}[1]{}{|V|\cdot\eta}{}{}
	\addConstraint{\sum_{i \in T} \eta}{\leq \sum_{ij \in \partial T} \frac 1{|E|}}{\text{ for each } T \subset \overline{S} \in \mathcal{\overline{S}}}
	\addConstraint{\eta}{\in \mathbb{R}}{{} }
  \end{maxi*}
  The optimal solutions of both problems are equal.
  The constraints of the projected problem are
  \[
  \eta \leq \frac {|\partial T|}{|T| |E|}
  \]
  for each $T \subset \overbar{S} \in \mathcal{\overbar{S}}$.
  The maximal value of $\eta$ is $\frac{\gamma(G, \mathcal{\overbar{S}})}{|E|}$
    so the maximal objective is indeed $\frac{|V|}{|E|} \gamma(G, \mathcal{\overbar{S}})$.
\end{proof}

The rest of the section is devoted to the study of Linear Problem~\ref{lp:uniform} in more depth.
We will give an alternate proof of Theorem~\ref{thm:cheeger} as well as an interpretation of the dual problem to Linear Problem~\ref{lp:uniform}.

\subsection{A single flow}

To simplify, we rescale the flows by $|E|$ so that the capacity of each edge is $1$ and we rescale the objective function by $\frac 1{|V|}$ to $\eta$, instead of $|V| \cdot \eta$.
The main reduction is to fix $k$ and solve the problem for a single column, i.e. for a single set $S = S_k$.
The reader is encouraged to use the supply table on page~\pageref{supply table} as a reference.
In the general case we compute the demand $\eta_i$ to be the minimum of the supplies in the $i$th row of the table.
Since we replaced each $\eta_i$ with a single value $\eta$, our goal is to maximize the minimal supply in the entire table.
Hence we may as well construct each flow $\varphi_k$ in such a way that the minimal value of the supply in the $k$th column is maximal.
This splits problem into independent problems that may be solved in each column separately.

\begin{lp}{Minimal isoperimetric number over $S$}\label{lp:isoperimetric single column}
	\begin{maxi*}[1]
		{}{\eta \hspace{8cm}}{}{}
		\addConstraint{\varphi_{ij}}{\leq 1}{\text{ for each } ij \in E}
		\addConstraint{-\varphi_{ij}}{\leq 1}{\text{ for each } ij \in E}
		\addConstraint{
			\sum_{j \in V, ji \in E} \varphi_{ji} -
			\sum_{j \in V, ij \in E} \varphi_{ij}
		}{ \geq \eta
		}{\text{ for each } i \in S}
		\addConstraint{\eta, \varphi_{ij}}{\in \mathbb{R}}{{} }
	\end{maxi*}
\end{lp}

Once again, to understand the problem we pass to its dual.

\begin{lp}{The dual to minimal isoperimetric number over $S$}\label{lp:isoperimetric single column dual}
	\begin{mini*}[1]
		{}{\sum_{ij \in E} |a_i - a_j| \hspace{4cm}}{}{}
		\addConstraint{\sum_{i \in S} a_i}{= 1}{}
		\addConstraint{a_i}{= 0}{\text{ for each } i \in V \setminus S}
		\addConstraint{a_i}{\geq 0}{\text { for each } i \in S}
	\end{mini*}
\end{lp}
\begin{theorem}
  Linear Problem~\ref{lp:isoperimetric single column dual} is dual to Linear Problem~\ref{lp:isoperimetric single column}.
\end{theorem}
\begin{proof}
  We derive the dual problem once again using the method from~\cite{lahaie2015} (with some shortcuts).
\[
  \eta \leq \eta + 
    \sum_{i \in S} a_i ( \sum_{ji \in E} \varepsilon_{ji} - \sum_{ij \in E} \varepsilon_{ij} - \eta ) + 
    \sum_{ij \in E} b^+_{ij} (1 - \varphi_{ij}) +
    \sum_{ij \in E} b^-_{ij} (1 + \varphi_{ij}) =
\]
\[
  \eta(1 - \sum_{i \in S} a_i) +
  \sum_{ij \in E} \varphi_{ij} (\delta_{j \in S} a_j - \delta_{i \in S} a_i - b_{ij}^+ + b_{ij}^-) +
  \sum_{ij \in E} b^+_{ij} + b^-_{ij}
\]
The inequality is true under the assumptions that
\[
  a_i \geq 0 \text{ and } b^+_{ij}, b^-_{ij} \geq 0.
\]
Without any loss of generality we may assume that $b^+_{ij} \cdot b^-_{ij} = 0$.
Then we have
\[
  b_{ij} = b^+_{ij} - b^-_{ij} \text{ and } |b_{ij}| = b^+_{ij} + b^-_{ij}
\]
and we have a bound on the right hand side 
\[
  \eta(1 - \sum_{i \in S} a_i) + \sum_{ij \in E} \varphi_{ij} (\delta_{j \in S} a_j - \delta_{i \in S} a_i  - b_{ij}) +
  \sum_{ij \in E} |b_{ij}| \leq \sum_{ij \in E} |b_{ij}|
\]
under the constraints
\[
  \sum_{i \in S} a_i = 1
\]
\[
  b_{ij} = \delta_{j \in S} a_j - \delta_{i \in S} a_i.
\]
Therefore the dual problem may be rewritten as Linear Problem~\ref{lp:isoperimetric single column dual}.
\end{proof}

To finish our argument we show that there always exists a simple optimal solution of Linear Problem~\ref{lp:isoperimetric single column dual}.

\begin{theorem}
  There exists an optimal solution of Linear Problem~\ref{lp:isoperimetric single column dual} with all non-zero values equal.
\end{theorem}
\begin{proof}
  Let 
  \[
   a = (a_i)_{i \in V} \in \mathbb{R}^{V}
  \]
  be an admissible solution of linear problem~\ref{lp:isoperimetric single column dual}.
  We will show that if $a$ is an optimal solution with maximal number of zeroes, then all non-zero values of $a$ are equal.

  Let
  \[
    \max a = \max \left\{ a_i \colon i \in S \right\}, \quad 
    \minp a = \min \left\{ a_i \colon i \in S, a_i > 0 \right \}
  \]
  and
  \[
    U = \left\{ i \in S \colon a_i = \max a \right\}, \quad
    L = \left\{ i \in S \colon a_i = \minp a \right\}.
  \]
  
  Both $U$ and $L$ are non-empty. If $\max a = \minp a$ (i.e. $U = L$), then we are done.
  Assume that $\max a > \minp a$.
  
  Let $a'(\delta) \in \mathbb{R}^V$ be defined by the following formula.
  \[
    a'_i(\delta) = \left\{
    \begin{array}{ll}
      a_i + \frac \delta{|U|} & \text{ for } i \in U \\
      a_i - \frac \delta{|L|} & \text{ for } i \in L \\
      a_i & \text{ otherwise }
    \end{array}
    \right.
  \]

  Since $S$ is finite, there exists $\Delta > 0$ such that the value of
  \[
    \sum_{ij \in E} |a'_i - a'_j| - \sum_{ij \in E} |a_i - a_j|
  \]
  is linear for $\delta \in [-\Delta, |L| \minp a]$.
  Since $a$ is an optimal solution, the value has to be constant.
  But then, $a'(|L| \minp a)$ is an optimal solution with a larger number of zeroes.
  
  Therefore if $a$ is an optimal solution with maximal number of zeroes, then all non-zero values of $a$ are equal to $\frac 1{|\supp a|}$.
\end{proof}

Let $a$ be an optimal solution of Linear Problem~\ref{lp:isoperimetric single column dual} with all non-zero values equal.
Then $\sum_{ij \in E} |a_i - a_j|$ is equal to the isoperimetric number of $\supp a \subset S$.
Therefore 
\[
\text{ optimal value } \geq \min_{T \subset S} \varphi(T).
\]

On the other hand let $T \subset S$ with minimal $\varphi(T)$. Let
\[
a_i = \frac 1{|T|} \text{ for } i \in T \text{ and } a_i = 0 \text{ otherwise.}
\]
We have
\[
\text{ optimal value } \leq \sum_{ij \in E} |a_i - a_j| = \frac 1{|T|} \cdot | \partial T | = \varphi(T) = \min_{T \subset S} \varphi(T).
\]
Therefore Linear Problem~\ref{lp:isoperimetric single column dual} is equivalent to Optimization Problem~\ref{op:isoperimetric number}.

\begin{op}{Minimal isoperimetric number over $S$}\label{op:isoperimetric number}
	Maximize $\eta$ such that for each $T \subset S$ we have
	\[
	\eta \leq \frac{|\partial T|}{|T|}.
	\]
\end{op}

The optimal solution of Linear Problem~\ref{lp:uniform} is equal to the minimum over $S_k$ of the solutions of Linear Problems~\ref{lp:isoperimetric single column} multiplied by $\frac{|V|}{|E|}$.
Hence Linear Problem~\ref{lp:uniform} has optimal solution equal to $\frac{|V|}{|E|} \gamma(G, \mathcal{S})$, as we have already proven in Theorem~\ref{thm:cheeger}.

\subsection{Mean property A}

It is very interesting to look at the dual problem to Linear Problem~\ref{lp:uniform}.
Since Linear Problem~\ref{lp:uniform} is the dual problem to a linearization of property A with some extra added constraints, its dual will be a relaxation of a linearization of property A.
We call this relaxation \df{mean property~A}. 

\begin{lp}{Mean property A at scale $\mathcal{S}$}\label{lp:mean}
\begin{mini}[1]
{}{\sum_{ij, \in E, k \in V} c_{ij, k} \hspace{4cm}}{}{}
\addConstraint{n_{i,j}}{=0}{\text{ for each } j \in V \setminus S_i}
\addConstraint{\sum_{i \in V, j \in S_i} n_{i,j}}{= |V|}{}
\addConstraint{n_{j,k} - n_{i,k}}{\leq c_{ij,k}}{\text{ for each } ij \in E, k \in V}
\addConstraint{c_{ij, k}, n_{i,j}}{\geq 0}{{} }
\end{mini}
\end{lp}

\begin{lemma}
  Linear Problem~\ref{lp:mean} is dual to Linear Problem~\ref{lp:uniform}.
\end{lemma}
\begin{proof}
  We derive the dual problem once again using the method from~\cite{lahaie2015} (with some shortcuts).
\[
  |V|\cdot\eta \leq |V|\cdot\eta + 
    \sum_{ij \in E, k \in V} (\frac 1{|E|} - \varphi_{ij,k})c^+_{ij,k}
    \sum_{ij \in E, k \in V} (\frac 1{|E|} + \varphi_{ij,k})c^-_{ij,k} +
\]
\[
  + \sum_{i \in V, j \in S_i} (\sum_{k \in V, ki\in E} \varphi_{ki,j} - \sum_{k \in V, ik \in E} \varphi_{ik,j} - \eta) n_{i,j} = \frac 1{|E|} \sum_{ij \in E, k \in V} (c^+_{ij,k} + c^-_{ij,k}) +
\]
\[
+ \eta (|V| - \sum_{i \in V, j \in S_i} n_{i,j}) + \sum_{ij \in E, k \in V} \varphi_{ij,k}
(-c^+_{ij,k} + c^-_{ij,k} + \delta_{k \in S_j} n_{j,k} - \delta_{k \in S_i} n_{i,k}) \leq
\]
\[
\leq \frac 1{|E|} \sum_{ij \in E, k \in V} (c^+_{ij,k} + c^-_{ij,k}).
\]

We substitute
\[
c_{ij,k} = c^+_{ij,k} - c^-_{ij,k},
\]
\[
|c_{ij,k}| = c^+_{ij,k} + c^-_{ij,k}
\]
and set
\[
  n_{i,k} = 0
\]
for $k \in V \setminus S_i$.

This gives us a problem equivalent to the dual problem.
\begin{mini}[1]
{}{\sum_{ij, \in E, k \in V} |c_{ij, k}| \hspace{4cm}}{}{\label{lp:dual equal capacities}}
\addConstraint{n_{i,j}}{=0}{\text{ for each } j \in V \setminus S_i}
\addConstraint{\sum_{i \in V, j \in S_i} n_{i,j}}{\geq |V|}{}
\addConstraint{n_{j,k} - n_{i,k}}{\leq c_{ij,k}}{\text{ for each } ij \in E, k \in V}
\addConstraint{c_{ij, k} \in \mathbb{R}, n_{i,j}}{\geq 0}{{} }
\end{mini}

Observe that we can always find an optimal solution with each $c_{ij, k} \geq 0$ and $\sum_{i,j \in V} n_{i,j} = |V|$. Therefore this problem is equivalent to Linear Problem~\ref{lp:mean}.  
\end{proof}

\begin{theorem}\label{thm:mean A}
Take a solution of Linear Problem~\ref{lp:mean} and let
\[
  \varphi_i(j) = n_{i,j}, \quad \varepsilon = \frac 1{|E||V|} \sum_{ij \in E, k \in V} c_{ij,k}.
\]
Then the following conditions are satisfied.
\begin{enumerate}
    \item $\{ \varphi_i \}$ has norm $1$ \emph{on average}, i.e.
    \[
      \frac 1{|V|} \| \sum_{i \in V} \varphi_i \|_1 = 1.
    \]
    \item $\{ \varphi_i \}$ has $\varepsilon$-variation \emph{on average}, i.e.
    \[
      \frac 1{|E|} \sum_{ij \in E} \sum_{k \in V} | \varphi_i(k) - \varphi_j(k)| = \frac 1{|E|} \sum_{ij \in E} \| \varphi_i - \varphi_j \|_1 \leq \varepsilon.
    \]
    \item $\supp \varphi_i \subset S_i$ for each $i \in V$.
\end{enumerate}
\end{theorem}

Note that if a graph has an isolated vertex $i$ and $S_i = \{ i \}$, then the optimal objective function value for this problem is clearly equal to $0$. This shows that the invariant in general is not suitable for use in geometric group theory (at least as far as applications to the Novikov conjecture go).

The known examples of Cayley graphs of finitely generated groups that do not have property A contain expanding sequences of subgraphs with Cheeger constants bounded away from $0$.
Hence these are examples of graphs without mean property~A.
This shows that there is plenty of room for examples of graphs without property~A that do not contain expanding subsequences.

\subsection{Sparsest cut}\label{sec:sparsest cut}

Let $G = (V, E)$ be a graph with capacity function $\kappa \colon E \to \mathbb{R}$.
The sparsity of a cut $(S, V \setminus S)$  in $G$ equals
\[
  \phi(S) = \frac{\sum_{ij \in \partial S} \kappa_{ij}}{\min(|S|, |V\setminus S|)}.
\]
The \df{sparsest cut} problem is to determine the value
\[
  \phi^\ast(\kappa) = \min \left\{ \frac{|\sum_{ij \in S} \kappa_{ij}|}{|S|} \colon S \subset V, S \neq \emptyset, |S| \leq \frac{|V|}2 \right\}.
\]
If $\kappa = \frac 1{|E|}$, then $\phi^\ast(\kappa)$ is equal to the Cheeger constant.

\begin{definition}
  Let $G = (V, E)$ be a graph with capacity function $\kappa \colon E \to \mathbb{R}$.
  Let $\mathcal{S}$ be a scale on $G$. We let
  \[
    \phi^\ast(\kappa, \mathcal{S}) = \min \left\{ \frac{|\sum_{ij \in T} \kappa_{ij}|}{|T|} \colon T \subset S \in \mathcal{S}, T \neq \emptyset \right\}
  \]
  be the \df{sparsest cut} of $G$ with capacity $\kappa$ at scale $\mathcal{S}$.
\end{definition}

\begin{theorem}
  The optimal solution of Linear Problem~\ref{lp:pseudo-flows} with uniform demands $\eta_i = \eta$ for each $i \in V$ is equal to
  \[
    \frac{|V|}{|E|} \max \left\{ \phi^\ast(\kappa, \mathcal{S}) \colon \sum_{ij \in E} \kappa_{ij} \leq 1 \right\},
  \]
  i.e. to the maximal sparsest cut value over all capacities with total capacity at most~$1$, rescaled by $\frac{|V|}{|E|}$.
\end{theorem}
The proof is very similar to the proof of Theorem~\ref{thm:cheeger} and is omitted.

\subsection{Computational complexity}

Computation of the Cheeger constant of a graph is NP-hard~\cite{leightonrao1999}. We observe that the Cheeger constant {\bf at scale $\mathcal{S}$} is computable in polynomial time.

\begin{theorem}
  Let $G$ be a finite graph and let $\mathcal{S}$ be a scale on $G$.
  The Cheeger constant of $G$ at scale $\mathcal{S}$ can be computed in
    polynomial time with respect to the size of $G$.
\end{theorem}
\begin{proof}
  As shown in the proof of Theorem~\ref{thm:cheeger}, the solution of the following linear problem
  \begin{maxi*}[1]{}{|V|\cdot\eta}{}{}
	\addConstraint{\sum_{i \in T} \eta}{\leq \sum_{ij \in \partial T} \frac 1{|E|}}{\text{ for each } T \subset S \in \mathcal{S}}
	\addConstraint{\eta}{\in \mathbb{R}}{{} }
  \end{maxi*}  
  is equal to $\frac{|V|}{|E|}\gamma(G, \mathcal{S})$.
  This problem is equivalent to Linear Problem~\ref{lp:isoperimetric} with extra constraints
  \[
    \eta_i = \eta \text{ for each } i \in V,
  \]
  \[
    \kappa_{ij} = \frac 1{|E|} \text{ for each } ij \in E.
  \]
  By Theorem~\ref{thm:lift and project}, it this problem is equivalent to Linear Problem~\ref{lp:pseudo-flows} with these extra constraints.
  
  The size of Linear Problem~\ref{lp:pseudo-flows} is polynomial in the size of $G$ (regardless of the choice of $\mathcal{S}$) and so is the number of extra constraints.
  The size of the coefficients depends only on $|E|$.
  Therefore the problem can be solved in time that is polynomial with
  respect to the size of $G$.
\end{proof}

A similar argument may be used to show that the other optimization problems at scale $\mathcal{S}$ that are discussed in this paper (i.e. property A and sparsest cut) are also polynomial.

The polynomial complexity of the Cheeger constant at scale $\mathcal{S}$ depends on the fact that we compute the minimum of $\frac{|\partial T|}{|T|}$ over a family of sets $\{ T \}$ that has a polynomial set of maximal elements (with respect to inclusion $\subset$). 
This is not the case when we consider the family $\{ T \colon |T| \leq \frac{|V|}2 \}$ and we cannot compute the Cheeger constant of the graph in polynomial time this way.

\section{The cube graph \texorpdfstring{- $\mathbb{Z}_2 \times \mathbb{Z}_2 \times \cdots \times \mathbb{Z}_2$}{}}

Let $Q_n$ be the graph with vertex set $\{ 0, 1 \}^n$
and Hamming distance between vertices, i.e.
\[
  d(u, v) = | \{ i \colon 1 \leq i \leq n, u_i \neq v_i \} |.
\]
Two vertices are connected by an edge iff they differ at exactly one coordinate.
We call~$Q_n$ the \df{$n$-dimensional hypercube} graph.
It is a Cayley graph of
\[
\underbrace{\mathbb{Z}_2 \times \mathbb{Z}_2 \times \cdots \times \mathbb{Z}_2}_{n \text{ times}}
\]
taken with the standard generating set.

We have already computed $\varepsilon_{1, Q_2} = \frac 23$ and $\varepsilon_{2, Q_3} = \frac 27$. 
Presently we give the general solution.

\begin{theorem}\label{thm:cube}
  Let $Q_n$ be the $n$-dimensional hypercube graph.
  The minimal variation of probability measures for $Q_n$ at scale $S \geq 0$ is
  \[
  \varepsilon_{S, Q_n} = \frac{2\binom{n-1}S}{\sum_{k=0}^S \binom{n}{k}}.
  \]
\end{theorem}

As an application we prove the following theorem of P. Nowak.
\begin{corollary}[P. Nowak, \cite{nowak2007}]
The disjoint union
\[
  \coprod_{n \in \mathbb{N}} \{ 0, 1 \}^n
\]
with $\ell_1$ metric does not have property A.
\end{corollary}
\begin{proof}
The space $\coprod_{n \in \mathbb{N}} \{ 0, 1 \}^n$ with $\ell_1$ metric is isometric to the disjoint union $G = \coprod_{n \in \mathbb{N}} Q_n$ of hypercube graphs.
We let 
\[G_n = \coprod_{m=1}^n Q_m.\]
By Theorem~\ref{thm:cube}, we have 
\[
\varepsilon_{S, G_n} = \max_{1 \leq m \leq n} \frac{2\binom{m-1}S}{\sum_{k=0}^S \binom{m}{k}}.
\]
For each $S \geq 0$ we have
\[
\limsup_{n \to \infty} \varepsilon_{S, G_n} \geq \lim_{n \to \infty} \frac{2\binom{n-1}S}{\sum_{k=0}^S \binom{n}{k}} = 2.
\]

Each of the graphs $G_n$ is a convex finite subset of $G$ and $G$ is locally finite.
By Theorem~\ref{thm:property A limit of finite graphs}, 
  $G$ does not have property A.
\end{proof}

We prove Theorem~\ref{thm:cube} by providing solutions of Optimization Problem~\ref{op:measures} (primal) and Optimization Problem~\ref{op:pseudo-flows} (dual) with equal objective values.

\subsection{Primal solution}

Let $Q_n = (V, E)$.
We have $V = \{ 0, 1 \}^n$ and $E$ is a set of pairs of points of $V$ with Hamming distance $1$. 
The edges are oriented so that the vertex with smaller norm is connected to the vertex with larger norm. We have
\[
  |V| = 2^n, \quad |E| = n 2^{n-1}.
\]
Let
\[
  B(i, r) = \{ j \in V \colon \|i-j\|_1 \leq r \},
  \quad
  S(i, r) = \{ j \in V \colon \|i-j\|_1 = r \}.
\]
We have
\[
  | B(i, r) | = \sum_{k = 0}^r \binom nk,
  \quad
  | S(i, r) | = \binom nr.
\]
The scale on $Q_n$ is
\[
  \mathcal{S} = \{ S_i = S(i, S) \colon i \in V \}.
\]
As a solution of Optimization Problem~\ref{op:measures}, we let
\[
  \xi_i = \frac {\chi_{B(i, S)}}{\|\chi_{B(i, S)}\|_1} = \frac 1{\sum_{k = 0}^S \binom nk} \chi_{S_i},
\]
where $\chi_A \colon V \to \mathbb{R}$ is a characteristic function of a set $A \subset V$.
Then for $ij \in E$ we have
\[
  | (S_i \cup S_j) \setminus (S_i \cap S_j) | = 2 \binom{n-1}{S}.
\]
Therefore
\[
  \| \xi_i - \xi_j \|_1 = \frac{2\binom{n-1}S}{\sum_{k=0}^S \binom{n}{k}}
\]
and
\[
  \varepsilon = \frac{2\binom{n-1}S}{\sum_{k=0}^S \binom{n}{k}}.
\]

\subsection{Dual solution} \label{subsection:dual-cubes}

The idea of the construction is very simple: we flow the maximum amount through the boundary $\partial S_j$ of $S_j \in \mathcal{S}$ and then spread the amount evenly among all nodes of $S_j$. We have to verify that there is enough capacity to carry this out.

We set equal capacities on edges
\[
  \kappa_{ij} = \frac{1}{|E|} = \frac 1{n2^{n-1}}, \text{ for each } ij \in E.
\]
We will construct a pseudo-flow for $j \in V$ on $S_j$ with supplies
\[
\sigma_{j, i} = \frac{\binom{n}{S+1}(S+1)}{n2^{n-1}\sum_{l=0}^S \binom{n}{l} }.
\]
These values do not depend on $j$ as the construction is analogous at each vertex.
To fix edge orientation we will describe it for $j = 000\ldots0$.

Consider spheres of radius $m$ and $m+1$ centered at $j$.
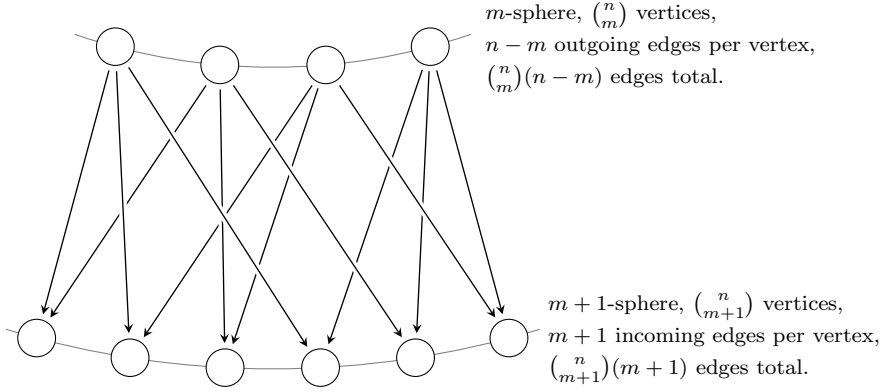
\begin{center}
\begin{tikzpicture}[rotate=270]
  \draw[gray] (17:12) arc (17:-17:12);
  \draw[gray] (19:8) arc (19:-19:8);

  \draw (15:8) node[vertex]  (1000) {};
  \draw (5:8) node[vertex]  (0100) {};
  \draw (-5:8) node[vertex]   (0010) {};
  \draw (-15:8) node[vertex] (0001) {};
  
  \draw (15:12) node[vertex] (1100) {};
  \draw (9:12) node[vertex] (1010) {};
  \draw (3:12) node[vertex] (1001) {};
  \draw (-3:12)  node[vertex] (0110) {};
  \draw (-9:12) node[vertex] (0101) {};
  \draw (-15:12) node[vertex] (0011) {};
  
  \draw (1000) edge[edge] (1100);
  \draw (1000) edge[edge] (1010);
  \draw (1000) edge[edge] (1001);
  \draw (0100) edge[edge] (1100);
  \draw (0100) edge[edge] (0110);
  \draw (0100) edge[edge] (0101);
  \draw (0010) edge[edge] (1010);
  \draw (0010) edge[edge] (0110);
  \draw (0010) edge[edge] (0011);
  \draw (0001) edge[edge] (1001);
  \draw (0001) edge[edge] (0101);
  \draw (0001) edge[edge] (0011);

  \node[draw=none, xshift=29mm, align=left] at (1000) { 
    \footnotesize $m$-sphere, $\binom{n}{m}$ vertices,\\
    \footnotesize $n-m$ outgoing edges per vertex, \\
    \footnotesize $\binom{n}{m}(n-m)$ edges total.
  };
  \node[draw=none, xshift=27mm, align=left] at (1100) { 
    \footnotesize $m+1$-sphere, $\binom{n}{m+1}$ vertices,\\
    \footnotesize $m+1$ incoming edges per vertex,\\
    \footnotesize $\binom{n}{m+1}(m+1)$ edges total.
  };
\end{tikzpicture}
\end{center}
\vspace{3mm}

We have
\[
  \binom{n}{m}(n-m) = \binom{n}{m+1}(m+1)
\]
edges connecting vertices of the $m$-sphere to vertices of the $m+1$-sphere.
We set the pseudo-flow on outgoing edges from the $S$-sphere to maximal capacity
\[
  \varphi_{j, ik} = -\frac{1}{|E|}, \text{ for each } i, k \in V, \|i-j\|_1 = S, \|k-j\|_1 = S+1.
\]
This transfers
\[
  \binom{n}{S+1}(S+1) \frac{1}{|E|} = \binom{n}{S+1} \frac{S+1}{n2^{n-1}}
\]
total weight into $S_j$. The flows on the inside edges will transfer supply internally and will not change the total supply. Our goal is to redistribute supply equally among all vertices of $S_j$, giving the values of $\sigma_{j, i}$ specified above.

To distribute the supply equally we need to transfer 
\[
  \frac{\binom{n}{S+1}(S+1)}{n2^{n-1}\sum_{l=0}^S \binom{n}{l}} \sum_{k=0}^m \binom{n}{k}
\]
of supply into an $m$-ball centered at $j$. This has to be done using $\binom {n}{m+1} (m+1)$ edges connecting $m$-sphere to $m+1$-sphere. We cannot exceed edge capacity and it is possible if and only if
\[
  \frac{
    \frac{\binom{n}{S+1}(S+1)}{n2^{n-1}\sum_{l=0}^S \binom{n}{l}} \sum_{k=0}^m \binom{n}{k}
  }
  { \binom{n}{m+1} (m+1)}
  \leq
  \frac 1{n2^{n-1}}.
\]
We can cancel $n2^{n-1}$ and rewrite the inequality as
\[
  \frac{\binom{n}{S+1}(S+1)}{\sum_{k=0}^S \binom{n}{k}} \leq 
  \frac{\binom{n}{m+1}(m+1)}{\sum_{k=0}^m \binom{n}{k}}.
\]
We have to prove that it is satisfied for each $m \leq S < n$. Let
\[
  w_m = \frac{\binom{n}{m+1}(m+1)}{\sum_{k=0}^m \binom{n}{k}}.
\]
It is enough to prove that
\[
  w_{m+1} \leq w_m
\]
for each $m < n$.

After canceling numerators we get the inequality
\[
\frac{\sum_{k=0}^m \binom{n}{k}}{\sum_{k=0}^{m+1} \binom{n}{k}} \leq \frac{m+1}{n-m-1}
\]
which is equivalent to (for $m+1 \leq n-m-1$, otherwise the inequality is trivial)
\[
\sum_{k=0}^m \binom{n}{k} \leq \frac{m+1}{n-2m-2} \binom{n}{m+1}.
\]
% https://mathoverflow.net/questions/17202/sum-of-the-first-k-binomial-coefficients-for-fixed-n
We have
\[
\frac{\binom{n}{m} + \binom{n}{m-1} + \cdots}{\binom{n}{m+1}} =
\frac{m+1}{n-m} + \frac{(m+1)m}{(n-m)(n-m+1)} + \cdots \leq
\]
\[
\leq \frac{m+1}{n-m} + \left(\frac{m+1}{n-m}\right)^2 + \cdots \leq 
\frac{m+1}{n-m} \cdot \frac{1}{1 - \frac{m+1}{n-m}} =
\]
\[
  = \frac{m+1}{n-2m-1} < \frac{m+1}{n-2m-2},
\]
which is the desired inequality.

Therefore we have 
\[
\sigma_{j, i} = \frac{\binom{n}{S+1}(S+1)}{n2^{n-1}\sum_{l=0}^S \binom{n}{l} }
\]
and
\[
\eta_i = \min_{j \in S_i} \eta_{j, i} = \frac{\binom{n}{S+1}(S+1)}{n2^{n-1}\sum_{l=0}^S \binom{n}{l}}.
\]
This gives a total demand
\[
\sum_{i \in V} \eta_i = 2^n \frac{\binom{n}{S+1}(S+1)}{n2^{n-1}\sum_{l=0}^S \binom{n}{l}}
  = \frac{2\binom{n}{S+1}(S+1)}{n\sum_{l=0}^S\binom{n}{l}}.
\]

Finally, we have
\[
\frac{2\binom{n}{S+1}(S+1)}{n\sum_{l=0}^S\binom{n}{l}} =
\frac{2\binom{n-1}S}{\sum_{k=0}^S \binom{n}{k}},
\]
because $\binom{n}{S+1} \frac{S+1}{n} = \binom{n-1}{S}$.

\subsection{Example \texorpdfstring{$n=3$, $S=1$}{n = 3, S = 1}}

A primal problem solution is obtained by setting
\[
  \xi_i = \frac 14 \chi_{S_i}.
\]
Then for each $ik \in E$
\[
  \| \xi_i - \xi_k \|_1 = 1
\]
and the value of the objective function is $\varepsilon = 1$.

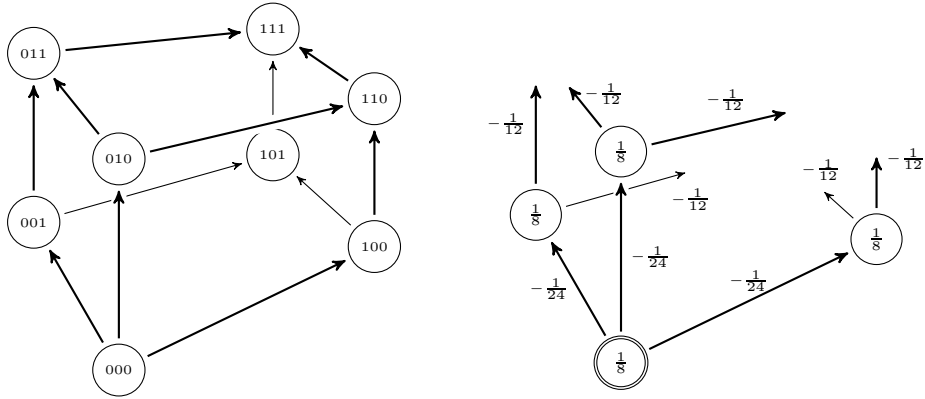
\begin{center}
% http://www.texample.net/tikz/examples/cuboid/
\begin{minipage}{0.48\textwidth}
\begin{tikzpicture}[->,>=stealth', shorten >=2pt, shorten <= 2pt, node distance=2.8cm, scale=1.4]
	\coordinate (P1) at (-4cm,5cm); % left vanishing point (To pick)
	\coordinate (P2) at (8cm,1.9cm); % right vanishing point (To pick)

	%% (A1) and (A2) defines the 2 central points of the cuboid
	\node[draw, circle] (010) at (0em,0cm) {$\fs 010$}; % central top point (To pick)
	\node[draw, circle] (000) at (0em,-2cm) {$\fs 000$}; % central bottom point (To pick)

	%% (A3) to (A8) are computed given a unique parameter (or 2) .8
	% You can vary .8 from 0 to 1 to change perspective on left side
	\node[draw, circle] (001) at ($(P1)!.8!(000)$) {$\fs 001$}; % To pick for perspective 
	\node[draw, circle] (011) at ($(P1)!.8!(010)$) {$\fs 011$};

	% You can vary .8 from 0 to 1 to change perspective on right side
	\node[draw, circle] (100) at ($(P2)!.7!(000)$) {$\fs 100$};
	\node[draw, circle] (110) at ($(P2)!.7!(010)$) {$\fs 110$};

	%% Automatically compute the last 2 points with intersections
	\node[draw, circle] (111) at
	  (intersection cs: first line={(110) -- (P1)},
			    second line={(011) -- (P2)}) {$\fs 111$};
	\node[draw, circle] (101) at
	  (intersection cs: first line={(100) -- (P1)}, 
			    second line={(001) -- (P2)}) {$\fs 101$};

    \draw (001) edge (101);
    \draw (100) edge (101);
    \draw (101) edge (111);

    \draw[thick] (000) edge (100);
    \draw[-, white, line width=2mm] (000) edge (010);
    \draw[thick] (000) edge (010);
    \draw[thick] (000) edge (001);

    \draw (100)[thick] edge (110);

    \draw[-, white, line width=2mm] (010) edge (110);
    \draw (010)[thick] edge (110);
    \draw (010)[thick] edge (011);

    \draw (001)[thick] edge (011);

    \draw (011)[thick] edge (111);
    \draw (110)[thick] edge (111);
\end{tikzpicture}
\end{minipage}
\begin{minipage}{0.48\textwidth}
\begin{tikzpicture}[->,>=stealth', shorten >=2pt, shorten <= 2pt, node distance=2.8cm, scale=1.4]
	\coordinate (P1) at (-4cm,5cm); % left vanishing point (To pick)
	\coordinate (P2) at (8cm,1.9cm); % right vanishing point (To pick)

	%% (A1) and (A2) defines the 2 central points of the cuboid
	\node[draw, circle] (010) at (0em,0cm) {$\fs \frac 18$}; % central top point (To pick)
	\node[draw, double, circle] (000) at (0em,-2cm) {$\fs \frac 18$}; % central bottom point (To pick)

	%% (A3) to (A8) are computed given a unique parameter (or 2) .8
	% You can vary .8 from 0 to 1 to change perspective on left side
	\node[draw, circle] (001) at ($(P1)!.8!(000)$) {$\fs \frac 18$}; % To pick for perspective 
	\node[draw=none, circle] (011) at ($(P1)!.8!(010)$) {};

	% You can vary .8 from 0 to 1 to change perspective on right side
	\node[draw, circle] (100) at ($(P2)!.7!(000)$) {$\fs \frac 18$};
	\node[draw=none, circle] (110) at ($(P2)!.7!(010)$) {};

	%% Automatically compute the last 2 points with intersections
	\node[draw=none, circle] (111) at
	  (intersection cs: first line={(110) -- (P1)},
			    second line={(011) -- (P2)}) {};
	\node[draw=none, circle] (101) at
	  (intersection cs: first line={(100) -- (P1)}, 
			    second line={(001) -- (P2)}) {};

    \draw[shorten >=30pt] (001) edge node[below right]{$\fs -\frac1{12}$} (101);
    \draw[shorten >=20pt] (100) edge node[above]{$\fs -\frac1{12}$} (101);

    \draw[thick] (000) edge node[above]{$\fs -\frac1{24}$} (100);
    \draw[-, white, line width=2mm] (000) edge (010);
    \draw[thick] (000) edge node[right]{$\fs -\frac1{24}$} (010);
    \draw[thick] (000) edge node[left]{$\fs -\frac1{24}$} (001);

    \draw (100)[thick, shorten >=20pt] edge node[right]{$\fs -\frac1{12}$} (110);

    \draw[-, white, line width=2mm] (010) edge (110);
    \draw (010)[thick, shorten >=30pt] edge node[above left]{$\fs -\frac1{12}$} (110);
    \draw (010)[thick, shorten >=15pt] edge node[right]{$\fs -\frac1{12}$} (011);

    \draw (001)[thick, shorten >=10pt] edge node[left]{$\fs -\frac1{12}$} (011);
\end{tikzpicture}
\end{minipage}
\end{center}
The pseudo-flow $\varphi_{000}$ transfers equal amounts along the edges from the $(m+1)$-sphere to the $m$-sphere with uniform supply.
\[
  \sigma_{000,000} = \sigma_{0001, 001} = \sigma_{000, 010} = \sigma_{000, 001} = \frac 18.
\]
By constructing analogous pseudo-flows at other vertices we show that $\eta_i = \frac 18$ for each $i$. This gives an objective value
\[
  \sum_{i \in V} \eta_i = 8 \cdot \frac 18 = 1,
\]
which proves that the primal solution found above is optimal.

\section{Regular graphs with large girth}\label{sec:girth}

There are $29503$ connected simple graphs with $12$ edges (up to isomorphism).
For Linear Problem~\ref{lp:measures} at scale $S=1$, the largest objective function value among these graphs is $\varepsilon = \frac{32}{29}$.
This value is achieved by the following two graphs.
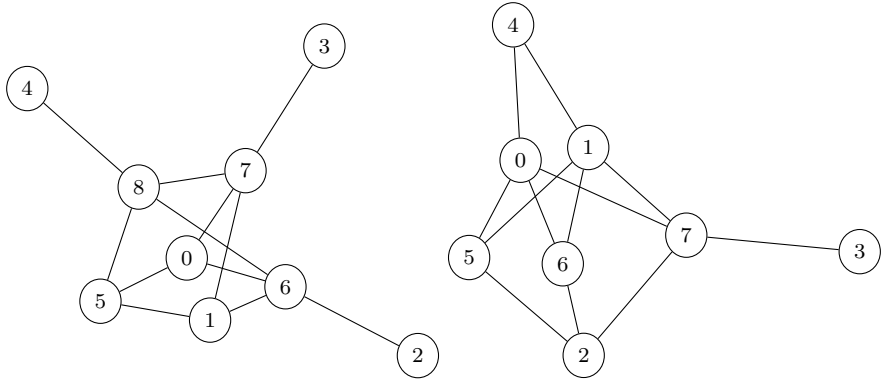
\begin{center}
\includegraphics[width=0.45\textwidth]{figures/girth1/graph.tikz}
\includegraphics[width=0.45\textwidth]{figures/girth2/graph.tikz}
\end{center}
Note that both graphs have girth $4$.
We recall that the \df{girth} of a graph $G$ is the length of a shortest cycle in $G$.
The $1$-balls in these graphs have no cycles.
At scale $S=2$, the largest objective function value is $\varepsilon = \frac 47$.
There are several graphs with this value and two of these graphs
  are shown below.
  
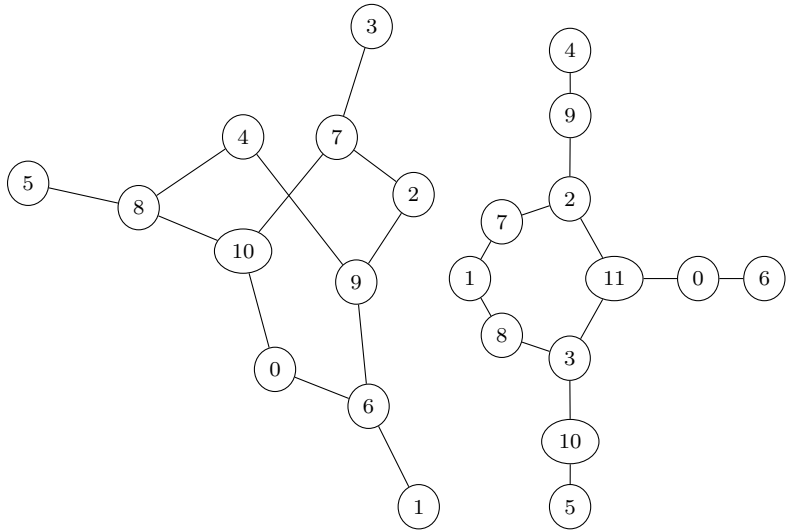
\begin{center}
\includegraphics[width=0.45\textwidth]{figures/girth3/graph.tikz}
\includegraphics[width=0.35\textwidth]{figures/girth4/graph.tikz}
\end{center}

Other graphs with $\varepsilon = \frac 47$ are similar and all have girth $6$.
The $2$-balls in these graphs have no cycles.
It suggests that large girth implies large variation of probability measures
  and we show that it is indeed the case in the following theorem.
\begin{theorem}\label{thm:girth}
	Let $G$ be a $d$-regular graph with girth $c$. Let $S \geq 0$. 
    If $d \geq 3$ and $2S + 1 < c$, then the minimal variation of probability measures for $G$ at scale $S$ is
	\[
	\varepsilon_{S, G} = \frac{2(d-1)^S(2-d)}{2-d(d-1)^S}.
	\]
\end{theorem}
Observe that the extremal graphs with $12$ edges shown above are not $d$-regular.
The Heawood graph is $3$-regular with girth $6$ and in Section~\ref{sec:heawood} we computed $\varepsilon=\frac 45$ at scale $S=2$ for this graph. 
It agrees with the value given in Theorem~\ref{thm:girth}.

As an application we prove the following theorem of R.~Willett (with slightly weaker assumptions in that we do not assume that sequence $d_n$ is bounded).
\begin{corollary}[R. Willett, \cite{willett2011}]
	Suppose $d_n$ is a sequence of integers with $d_n\ge 3$ and suppose $c_n$ is a sequence of integers going to infinity. 
	Let $G_n$ be a $d_n$-regular graph with girth $c_n$.
	The disjoint union 
	\[
	\coprod_{n \in \mathbb{N}} G_n
	\]
	fails to have property~A.
\end{corollary}
\begin{proof} 
Fix $S \geq 0$. 
Let $N \in \mathbb{N}$ be such that for each $n \geq N$ we have $c_n > 2S + 1$ .
From Theorem~\ref{thm:girth}, for each $n \geq N$, we have
\[
  \varepsilon_{S, G_n} = \frac{2(d_n-1)^S(2-d_n)}{2-d_n(d_n-1)^S}.
\]

Graph $H_n = \coprod_{i=1}^n G_n$ is a finite convex subgraph of $\coprod_{n \in \mathbb{N}} G_n$.
We have
\[
  \varepsilon_{S, H_n} \geq \max \left\{ \frac{2(d_i-1)^S(2-d_i)}{2-d_i(d_i-1)^S} \colon N \leq i \leq n \right\}.
\]

If $d_n$ is bounded, then there exists an increasing sequence $i_1, i_2, i_3, \ldots$, with $i_n \geq N$, such that $d_{i_j} = d_{i_k}$ for each $j, k \in \mathbb{N}$.
Then
\[
  \limsup_{n \to \infty} \varepsilon_{S, H_n} \geq \lim_{n \to \infty} \varepsilon_{S, H_{i_n}} \geq \frac{2(d-1)^S(2-d)}{2-d(d-1)^S},
\]
for $d \geq 3$ that does not depend on $S$.
Then
\[
  \lim_{S \to \infty} \limsup_{n \to \infty} \varepsilon_{S, H_n} \geq 
  \lim_{S \to \infty} \frac{2(d-1)^S(2-d)}{2-d(d-1)^S} = 2 - \frac 4d \geq 2 - \frac 43 = \frac 23.
\]
By Theorem~\ref{thm:property A limit of finite graphs}, $\coprod_{n \in \mathbb{N}} G_n$ does not have property~A.

If $d_n$ is unbounded, then 
\[
\limsup_{n \to \infty} \varepsilon_{S, H_n} \geq 
  \lim_{d \to \infty} \frac{2(d-1)^S(2-d)}{2-d(d-1)^S} = 2
\]
and
\[
\lim_{S \to \infty} \limsup_{n \to \infty} \varepsilon_{S, H_n} = 2.
\]
By Theorem~\ref{thm:property A limit of finite graphs}, $\coprod_{n \in \mathbb{N}} G_n$ does not have property A.
\end{proof}

\subsection{Cheeger constant of a regular graph with large girth}

We will compute weighted isoperimetric inequalities for subsets of $d$-regular trees that do not contain leaves of the tree.
This allows us to compute the Cheeger constant at scale $S$ for the graph from Theorem~\ref{thm:girth} and obtain a lower bound on the value of $\varepsilon_{S, G}$.

Fix $d\ge 3$. 
Let $T = (V, E)$ be a $d$-regular tree of depth $S+1$.
Let $r$ denote the root of $T$.
Let $U \subset V$ such that $U \subset B(r, S)$ ($U$ does not contain leaves of $T$).
Let $n = |U|$. Assume that $U$ has $k$ connected components.
Since $T$ is $d$-regular, we have
\[
  | \partial U | = d n - 2(n - k)
\]
Therefore
\[
\varphi(U) = \frac{|\partial U|}{|U|} = \frac{(d-2)n + 2k}{n}.
\]
We have
\[
  \min_U \varphi(U) = d-2 + 2\frac 1n,
\]
for $n = |B(r, S)| = \frac{2 - d(d-1)^S}{2-d}$.
Therefore
\[
  \min_{U \subset B(r, S)}  \frac{|\partial U|}{|U|} = d - 2 + \frac {2(2-d)}{2 - d(d-1)^S}
\]

Let $G$ be a $d$-regular graph with girth $c$ and let $S \geq 0$ such that $2S+1 < c$.
Hence the Cheeger constant of $G$ at scale $S$ is equal to
\[
\gamma(G, S) = d - 2 + \frac {2(2-d)}{2 - d(d-1)^S} =
  (2-d) (\frac 2{2- d(d-1)^S} - 1) =
  \frac {(2-d)d(d-1)^S}{2 - d(d-1)^S}.
\]

By Theorem~\ref{thm:cheeger}, we have
\[
\varepsilon_{S, G} \geq \frac{m}{\frac{dm}2} \frac {(2-d)d(d-1)^S}{2 - d(d-1)^S} =
  \frac{2(2-d)(d-1)^S}{2 - d(d-1)^S}.
\]

\subsection{The primal solution}

Let $G = (V, E)$ be a $d$-regular graph with girth $c$.
Let $|V| = n$. We have $|E| = \frac{nd}2$. 
For each $i \in V$ we have $S_i = B(i, S)$.
Since $G$ is $d$-regular and $2S + 1 < c$, we have
\[
|S_i| = \frac{2-d(d-1)^S}{2-d} \text{ for each } i \in V
\]
and
\[
| (S_i \setminus S_j) \cup (S_j \setminus S_i) | = 2(d-1)^S \text{ for each } ij \in E.
\]

To obtain the primal solution, set
\[
\xi_i = \frac{\chi_{S_i}}{|S_i|}.
\]
For each $ij \in E$ we have
\[
\| \xi_i - \xi_j \|_1 = 
  | (S_i \setminus S_j) | \frac 1{|S_i|} + 
  | (S_j \setminus S_i) | \frac 1{|S_j|} =
  \frac{2(d-1)^S(2-d)}{2-d(d-1)^S}.
\]

This shows that
\[
\varepsilon_{S, G} \leq \frac{2(d-1)^S(2-d)}{2-d(d-1)^S},
\]
which completes the proof of Theorem~\ref{thm:girth}.

\subsection{The curious case of infinite trees}
At first glance, it appears that the argument above proves that infinite regular trees fail to have property~A, which is false. 
We wish to explain why this argument does not prevent trees from having property~A. 
We observe that in regular graphs with large girth, every vertex ``looks like'' the 
root of a regular tree and so for any ball $B$, one can 
estimate $|\partial B|/|B|$ as above.
On the other hand, to apply our argument to an infinite tree, we have 
to restrict our attention to a finite portion of it, using Theorem~\ref{thm:property A limit of finite graphs}.
For a ball $B$ around a vertex near 
the leaves of 
such a tree (see Figure~\ref{fig:trees}), it is not possible to bound the 
ratio $|\partial B|/|B|$ away from $0$.

\begin{figure}
  \centering
  \includegraphics[width=\textwidth]{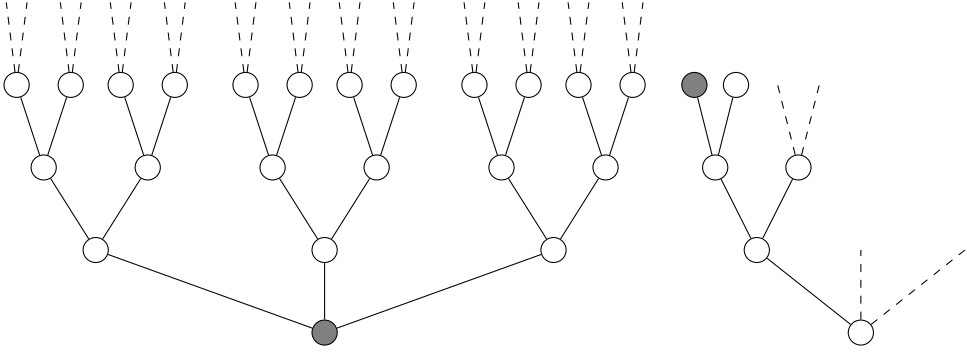}
  \caption{The ball of radius $3$ based at the root of the subset $S$ a tree (left) compared to a ball based at a leaf (right).}
  \label{fig:trees}
\end{figure}

\section{Proof of Theorem~\ref{thm:property A limit of finite graphs}}

\begin{lemma}\label{lem:subgraph}
  Let $G$ be a (possibly infinite) graph.
  Let $H$ be a convex finite subgraph of $G$.
  Let $S \geq 0$.
  We have
  \[
    \varepsilon_{2S, H} \leq \varepsilon_{S, G},
  \]
  where $\varepsilon_{2S, H}$ and $\varepsilon_{S, G}$ are optimal solutions
    of Optimization Problem~\ref{op:measures} for graphs $H$ and $G$ at scales $2S$ and $S$ respectively.
\end{lemma}
\begin{proof}
First we observe that the weak duality part of Theorem~\ref{thm:duality} still holds even if the graph $G$ is infinite since the sums in question are absolutely convergent.

%Let $\xi, \varepsilon$ be an admissible solution of Optimization Problem~\ref{op:measures} for a (possibly infinite) graph $G$ at scale $S$.
%Let $\kappa, \eta, \varphi$ be an admissible solution of Optimization Problem~\ref{op:pseudo-flows} for $G$ at scale $S$.
%Since $\kappa_{ij} \geq |\varphi_{k, ij}|$, for each edge $ij$, we have
%\[
%\kappa_{ij} \cdot \varepsilon \geq \sum_{k \in V} \varphi_{k, ij} (\xi_j(k) - \xi_i(k)).
%\]
%The sum is absolutely convergent because $\| \xi_j - \xi_i \|_1 \leq \varepsilon < \infty$ and $\varphi_{k, ij}$ is bounded (by $\kappa_{ij}$).
%
%The condition $\sum_{ij\in E}\kappa_{ij}\le 1$ implies that for every $l\in V$,
%\[
%\varepsilon \ge \sum_{ij \in E} \sum_{k \in V} \varphi_{k, ij} (\xi_j(k) - \xi_i(k)) =
%\sum_{i \in V} \sum_{k \in V} \xi_i(k) \left( \sum_{mi \in E} \varphi_{k, mi} - \sum_{im \in E} \varphi_{k, im} \right) 
%\geq
%\]
%\[
%\geq \sum_{i \in V} \sum_{k \in V} \xi_i(k) \eta_i
%= \sum_{i \in V} \eta_i \sum_{k \in V} \xi_i(k) = \sum_{i \in V} \eta_i. 
%\]
%Once again each sum is absolutely convergent. In particular, for each $i \in V$, we have $\eta_i \leq \sum_{mi \in E} \kappa_{mi} + \sum_{im \in E} \kappa_{im}$, so $\sum_{i \in V} \max (0, \eta_i) \leq 2\sum_{ij \in E} \kappa_{ij} = 2$ and $\sum_{i \in V} \eta_i$ is convergent from the problem definition, which implies absolute convergence. 
%
Thus, if both $\varepsilon =\varepsilon_{S, G}$ and $\sum_{i \in V} \eta_i$ are optimal, then 
\[
  \varepsilon_{S, G} \geq \sum_{i \in V} \eta_i.
\]

Observe that the definition of Optimization Problem~\ref{op:pseudo-flows} may be extended so that the scale $\mathcal{S}$ is an arbitrary family of subsets of the vertex set, not necessarily indexed by vertices of the underlying graph. (This is not the case for the primal problem, where we want to measure the variance of probability measures on adjacent vertices.)

Let $H = (V_H, E_H)$ and let $\mathcal{S'} = \{ B_G(v, S) \cap V_H \}_{v \in V}$
be the scale $S$ on $G$ truncated to the subgraph $H$.
Let $\eta' \colon V_H \to \mathbb{R}, \kappa' \colon E_H \to \mathbb{R}, \{ \varphi'_i \colon E_H \to \mathbb{R} \}_{i \in V}$ be a solution of Optimization Problem~\ref{op:pseudo-flows} for the family $\mathcal{S'}$.
We may extend $\eta', \kappa', \varphi'$ to an admissible solution of Optimization Problem~\ref{op:pseudo-flows} for $G$ at scale $S$, by setting all missing values to $0$.
This shows that
\[
\sum_{i \in V} \eta_i \geq \sum_{i \in V} \eta'_i.
\]

To finish the argument we appeal to the projection of Optimization Problem~\ref{op:pseudo-flows} to Linear Problem~\ref{lp:isoperimetric} (for the finite graph $H$).
Clearly, Linear Problem~\ref{lp:isoperimetric} may be reformulated for an arbitrary family of subsets of $V_H$, not necessarily indexed by vertices of $H$.
Optimization Problem~\ref{op:pseudo-flows} for $H$ at scale $\mathcal{S'}$ is equivalent to Linear Problem~\ref{lp:isoperimetric} for the family $\mathcal{S'}$.
Clearly, if $\mathcal{S'}$ refines $\mathcal{S}$, then the optimal solution of Linear Problem~\ref{lp:isoperimetric} for $\mathcal{S'}$ will be greater than or equal to the solution of Optimization Problem~\ref{op:pseudo-flows} for $\mathcal{S}$ (we have more constraints in the second case).
If $\mathcal{S} = \{ B_H(v, 2S) \}_{v \in V_H}$, then $\mathcal{S'}$ refines $\mathcal{S}$ (since $H$ is a convex subgraph of $G$, the path-length metric on $H$ is identical to the path-length metric restricted from $G$).
Therefore
\[
\sum_{i \in V} \eta'_i \geq \varepsilon_{2S, H}.
\]
\end{proof}

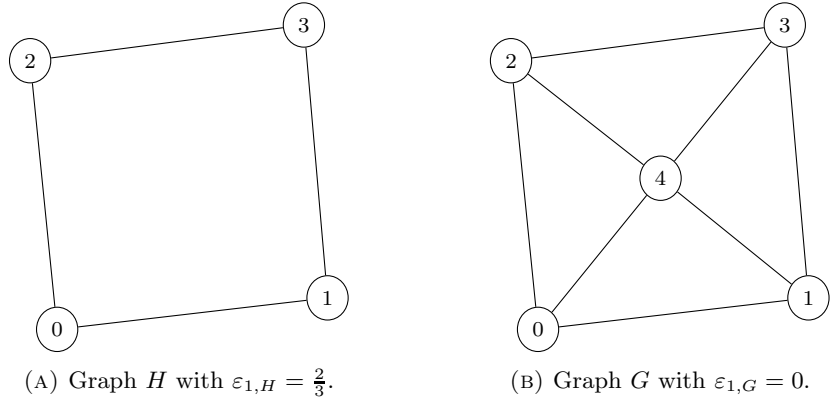
\begin{figure}[H]
\centering
\begin{subfigure}{.5\textwidth}
\centering
\begin{tikzpicture}[>=latex,line join=bevel,font=\footnotesize,scale=0.8]
  \node (0) at (39.901bp,17.989bp) [draw,ellipse] {0};
  \node (1) at (165.9bp,32.786bp) [draw,ellipse] {1};
  \node (2) at (27.307bp,144.22bp) [draw,ellipse] {2};
  \node (3) at (154.89bp,161.0bp) [draw,ellipse] {3};
  \draw [] (0) -- (1);
  \draw [] (0) -- (2);
  \draw [] (1) -- (3);
  \draw [] (2) -- (3);
\end{tikzpicture}
\caption{Graph $H$ with $\varepsilon_{1, H} = \frac 23$.}
\end{subfigure}%
\begin{subfigure}{.5\textwidth}
\centering
\begin{tikzpicture}[>=latex,line join=bevel,font=\footnotesize,scale=0.8]
  \node (0) at (39.901bp,17.989bp) [draw,ellipse] {0};
  \node (1) at (165.9bp,32.786bp) [draw,ellipse] {1};
  \node (2) at (27.307bp,144.22bp) [draw,ellipse] {2};
  \node (3) at (154.89bp,161.0bp) [draw,ellipse] {3};
  \node (4) at (96.9995bp,88.99875bp) [draw,ellipse] {4};
  \draw [] (0) -- (1);
  \draw [] (0) -- (2);
  \draw [] (1) -- (3);
  \draw [] (2) -- (3);
  \draw [] (0) -- (4);
  \draw [] (1) -- (4);
  \draw [] (2) -- (4);
  \draw [] (3) -- (4);
\end{tikzpicture}
\caption{Graph $G$ with $\varepsilon_{1, G} = 0$.}
\end{subfigure}
\caption{The increase of scale in subgraph $H$ of graph $G$ in Lemma~\ref{lem:subgraph} is necessary. Graph $H$ is a convex subgraph of a graph $G$ and it has larger $\varepsilon$ value at scale $S = 1$. This is caused by the set $\{ 0,1,2,3 \}$ from the scale $S=1$ on $G$ restricted to $H$ that is not in the scale $S=1$ of $H$.}
\end{figure}

We are now ready to prove Theorem~\ref{thm:property A limit of finite graphs}.
{
\renewcommand{\thetheorem}{\ref{thm:property A limit of finite graphs}}
\begin{theorem}
  Let $G$ be a countably infinite locally finite graph endowed with the path-length metric.
  Let $G_1 \subset G_2 \subset G_3 \subset \cdots$ be an ascending sequence of 
  convex finite subgraphs of $G$ such that $G = \bigcup_{n \in \mathbb{N}} G_n$.
  Let $\varepsilon_{S, G_n}$ be the minimal variation of probability
  measures at scale $S$ for graph~$G_n$, i.e. the optimal solution of Optimization Problem~\ref{op:measures} for $G_n$ at scale $S$.
  Graph $G$ has property $A$ iff
  \[
  \lim_{S \to \infty} \limsup_{n \to \infty} \varepsilon_{S, G_n} = 0.
  \]
\end{theorem}
\addtocounter{theorem}{-1}
}
\begin{proof}[Proof of Theorem~\ref{thm:property A limit of finite graphs}]
Let $G_n = (V_n, E_n)$. If $G$ has property A, then
\[
  \lim_{S \to \infty} \varepsilon_{S, G} = 0
\]
By Lemma~\ref{lem:subgraph}, for each $n$ we have
\[
  \varepsilon_{2S, G_n} \leq \varepsilon_{S, G}.
\]
Hence
\[
  0 \leq \lim_{S \to \infty} \limsup_{n \to \infty} \varepsilon_{S, G_n} \leq
    \lim_{S \to \infty} \lim_{n \to \infty} \varepsilon_{S/2, G} = 0.
\]

To prove the implication in the opposite direction, assume that
\[
  \lim_{S \to \infty} \limsup_{n \to \infty} \varepsilon_{S, G_n} = 0.
\]
Let $\{ \xi^n_i \colon V_n \to \mathbb{R} \}_{i \in V_n}$ be an optimal solution of Optimization Problem~\ref{op:measures} for $G_n$ at scale $S$. 
Using a standard diagonal argument we may pick a subsequence $n_1, n_2, \ldots$ such that for each $i, j \in V$ the limit
\[
  \xi_i(j) = \lim_{m \to \infty} \xi^{n_m}_i(j)
\]
exists (the sequence may be undefined for small values of $m$).
\begin{enumerate}
    \item $\xi_i(j) \geq 0$ since for large enough $n$ (once $i, j \in V_n$) we have $\xi^n_i(j) \geq 0$.
    \item $\supp \xi_i \subset B(i, S)$ since for each $j \not\in B(i, S)$ and large enough $n$ we have $\xi^n_i(j) = 0$ (from convexity of $G_n$, we have $B_{G_n}(i, S) = B_G(i, S) \cap V_n$).
    \item Since $\supp \xi_i$ is finite and $\supp \xi^{n_m}_i \subset \supp \xi_i$, we have
    \[
    \| \xi_i \|_1 = \| \lim_{m \to \infty} \xi^{n_m}_i \|_1 = \lim_{m \to \infty} \| \xi^{n_m}_i \|_1 = 1,
    \]
    for each $i \in V$.
    \item Similarly, for each $ij \in E$,
    \[
      \| \xi_i - \xi_j \|_1 = \| \lim_{m \to \infty} \xi^{n_m}_i - \lim_{m \to \infty} \xi^{n_m}_j \|_1 = \lim_{m \to \infty} \| \xi^{n_m}_i - \xi^{n_m}_j \|_1 \leq 
      \limsup_{m \to \infty} \varepsilon_{S, G_m}.
    \]
\end{enumerate}
Therefore $\{ \xi_i \}$ is a solution of Optimization Problem~\ref{op:measures} for $G$ at scale $S$ with $\varepsilon_S = \limsup_{n \to \infty} \varepsilon_{S, G_n}$.
By the assumption $\lim_{S\to \infty} \varepsilon_S = 0$, so $G$ has property A.
\end{proof}

\bibliographystyle{plain}
\bibliography{references.bib}
\end{document}